\documentclass[reqno]{amsart}

\usepackage[dvipdfmx,hidelinks]{hyperref}
\usepackage{amsmath,amssymb,amsthm,braket}

\newtheorem{theo}{Theorem}[section]
\newtheorem{lemm}[theo]{Lemma}
\newtheorem{corr}[theo]{Corollary}
\newtheorem{prop}[theo]{Proposition}

\numberwithin{equation}{section}

\theoremstyle{definition}
\newtheorem{defi}[theo]{Definition}

\newtheorem{rema}[theo]{Remark}

\newtheorem{convention}[theo]{Convention}

\newcommand{\dv}{{\rm div}\,}
\newcommand{\supp}{{\rm supp}\,}

\newcommand{\BR}{\mathbb{R}}
\newcommand{\BC}{\mathbb{C}}

\newcommand{\BN}{\mathbb{N}}

\newcommand{\BP}{\mathbb{P}}

\newcommand{\CA}{\mathcal{A}}
\newcommand{\CB}{\mathcal{B}}

\newcommand{\CD}{\mathcal{D}}

\newcommand{\CF}{\mathcal{F}}

\newcommand{\CL}{\mathcal{L}}

\newcommand{\CR}{\mathcal{R}}
\newcommand{\CS}{\mathcal{S}}
\newcommand{\CT}{\mathcal{T}}

\newcommand{\CX}{\mathcal{X}}

\newcommand{\pd}{\partial}
\renewcommand{\L}{\mathrm{L}}
\renewcommand{\H}{\mathrm{H}}
\newcommand{\B}{\mathrm{B}}
\newcommand{\W}{\mathrm{W}}
\newcommand{\C}{\mathrm{C}}
\newcommand{\loc}{\mathrm{loc}}
\renewcommand{\Re}{\mathrm{Re}}

\newcommand{\dx}{\,\mathrm{d}x}

\newcommand{\dt}{\,\mathrm{d}t}
\newcommand{\I}{\mathrm{Id}}

\newcounter{teller}

\begin{document}
\title[The Navier--Stokes equations in exterior Lipschitz domains: $\L^p$-theory]
{The Navier--Stokes equations in exterior Lipschitz domains: $\L^p$-theory}	
\author{Patrick Tolksdorf}
\author{KEIICHI WATANABE}
	
\address{Universit\'e Paris-Est Cr\'eteil, LAMA UMR CNRS 8050, France}
\email{patrick.tolksdorf@u-pec.fr}
\address{Department of Pure and Applied Mathematics, Graduate School of
Fundamental Science and Engineering, Waseda University,
3-4-1 Ookubo,Shinjuku-ku, Tokyo,169-8555, Japan}
\email{keiichi-watanabe@akane.waseda.jp}
\thanks{The first author was partially supported by the project ANR INFAMIE (Universit\'e Paris-Est Cr\'eteil) and the second author was supported by Grant-in-Aid for JSPS Fellows 19J10168 and Top Global University Project (Waseda University).
}

\maketitle
\begin{abstract}
We show that the Stokes operator defined on $\L^p_{\sigma} (\Omega)$ for an exterior Lipschitz domain $\Omega \subset \BR^n$ $(n \geq 3)$ admits maximal regularity provided that $p$ satisfies $| 1/p - 1/2| < 1/(2n) + \varepsilon$ for some $\varepsilon > 0$.
In particular, we prove that the negative of the Stokes operator generates a bounded analytic semigroup on $\L^p_\sigma (\Omega)$ for such $p$.
In addition, $\L^p$-$\L^q$-mapping properties of the Stokes semigroup and its gradient with optimal decay estimates are obtained. This enables us to prove the existence of mild solutions to the Navier--Stokes equations in the critical space $\L^{\infty} (0 , T ; \L^3_{\sigma} (\Omega))$ (locally in time and globally in time for small initial data).
\end{abstract}

\section{Introduction}
\noindent
Let $\Omega$ be an exterior Lipschitz domain in $\BR^n$ $(n\geq 3)$, {\em i.e.}, the complement of a bounded Lipschitz domain.
In this paper, we investigate the Stokes resolvent problem subject to homogeneous Dirichlet boundary conditions
\begin{align}
\label{Eq: Stokes resolvent problem}
\left\{\begin{aligned}
\lambda u - \Delta u + \nabla \pi & = f && \text{in } \Omega, \\
\dv (u) & = 0 && \text{in } \Omega, \\
u & = 0 && \text{on } \pd\Omega,
\end{aligned}\right.
\end{align}
where $u={}^\top\!(u_1, \dots, u_n) : \Omega \to \BC^n$ and $\pi : \Omega \to \BC$ are the unknown velocity field and the pressure, respectively. The right-hand side $f$ is supposed to be divergence-free and $\L^p$-integrable for an appropriate number $1 < p < \infty$ and the resolvent parameter $\lambda$ is supposed to be contained in a sector $\Sigma_{\theta} := \{ z \in \BC \setminus \{ 0 \} |\ \lvert \arg(z) \rvert < \theta \}$ with $\theta \in (0 , \pi)$. \par
In the case of bounded Lipschitz domains, resolvent bounds were proven by Shen~\cite{She} for numbers $p$ satisfying the condition
\begin{align}
\label{cond-p}
 \Big\lvert \frac{1}{p} - \frac{1}{2} \Big\rvert < \frac{1}{2 n} + \varepsilon,
\end{align}
where $\varepsilon > 0$ is a number that only depends on the dimension $n$, the opening angle $\theta$, and quantities describing the Lipschitz geometry. A corollary of Shen's result is that the negative of the Stokes operator generates a bounded analytic semigroup. This was an affirmative answer to a problem posed by Taylor in~\cite{T}. Recently, the study of the Stokes operator on bounded Lipschitz domains was continued by Kunstmann and Weis~\cite{KW} and the first author of this article~\cite{Tol, Tol-phd}. In~\cite{KW}, the property of maximal regularity and the boundedness of the $\H^{\infty}$-calculus were established yielding a short proof to reveal the domain of the square root of the Stokes operator as $\W^{1 , p}_{0 , \sigma} (\Omega)$, see~\cite{Tol}. \par
The purpose of this paper is to continue the study of the Stokes operator in the case of exterior Lipschitz domains $\Omega$. If $\Omega$ has a connected boundary, the existence of the Helmholtz decomposition was proven by Lang and M\'endez in~\cite{LM}. More precisely, Lang and M\'endez proved that the Helmholtz projection, {\em i.e.}, the orthogonal projection $\BP$ of $\L^2(\Omega ; \BC^n)$ onto $\L^2_{\sigma} (\Omega)$, defines a bounded projection from $\L^p(\Omega ; \BC^n)$ onto $\L^p_{\sigma}(\Omega)$ for $p \in (3 / 2 , 3)$. Here and below, $\L^p_{\sigma} (\Omega)$ denotes the closure of
\begin{align*}
 \C_{c , \sigma}^{\infty} (\Omega) := \{ \varphi \in \C_c^{\infty} (\Omega ; \BC^n) |\ \dv (\varphi) = 0 \}
\end{align*}
in $\L^p(\Omega ; \BC^n)$. As we will see in Subsection~\ref{Sec: The Helmholtz projection on exterior Lipschitz domains} --- and this will be crucial for the whole proof --- a short analysis of the proof of Lang and M\'endez shows the validity of this result for $p \in (3 / 2 - \varepsilon , 3 + \varepsilon)$ and some $\varepsilon > 0$. We will investigate the Stokes operator, which is defined on $\L^2_{\sigma} (\Omega)$ by using a sesquilinear form, see~\cite[Ch.~4]{MM}. On $\L^p_{\sigma} (\Omega)$ for $1 < p < \infty$, the Stokes operator $A_p$ is defined in two steps. First, take the part of $A_2$ in $\L^p_{\sigma} (\Omega)$, \textit{i.e.},
\begin{align*}
 \CD(A_2|_{\L^p_{\sigma}}) := \{ u \in \CD(A_2) \cap \L^p_{\sigma} (\Omega) : A_2 u \in \L^p_{\sigma} (\Omega) \}
\end{align*}
and $A_2|_{\L^p_{\sigma}} u$ is given by $A_2 u$ for $u$ in its domain. Notice that $A_2|_{\L^p_{\sigma}}$ is densely defined since $\C_{c , \sigma}^{\infty} (\Omega) \subset \CD(A_2|_{\L^p_{\sigma}})$ and that it is closable. In the second step, we define $A_p$ to be the closure of $A_2|_{\L^p_{\sigma}}$ in $\L^p_{\sigma} (\Omega)$. \par 
The first main result of this article is the following theorem, which includes an affirmative answer to Taylor's conjecture~\cite{T} in the case of exterior Lipschitz domains.

\begin{theo}
\label{th-1.1}
Let $\Omega$ be an exterior Lipschitz domain in $\BR^n$ $(n\geq 3)$. Then there exists $\varepsilon > 0$ such that for all numbers $p$ that satisfy~\eqref{cond-p} the Stokes operator $A_p$ is closed and densely defined, its domain continuously embeds into $\W^{1 , p}_{0 , \sigma} (\Omega)$, and $- A_p$ generates a bounded analytic semigroup $(T(t))_{t \geq 0}$ on $\L^p_{\sigma} (\Omega)$. Furthermore, for all $1 < p \leq q < \infty$ that both satisfy~\eqref{cond-p} the semigroup $T(t)$ maps for $t > 0$ the space $\L^p_{\sigma} (\Omega)$ continuously into $\L^q_{\sigma} (\Omega)$. Moreover, there exists a constant $C > 0$ such that
\begin{align}
\label{Eq: Lp-Lq estimates}
 \| T(t) f \|_{\L^q_{\sigma} (\Omega)} \leq C t^{- \frac{n}{2} (\frac{1}{p} - \frac{1}{q})} \| f \|_{\L^p_{\sigma} (\Omega)} \qquad (t > 0 , f \in \L^p_{\sigma} (\Omega)).
\end{align}
If $p$ and $q$ satisfy additionally $p \leq 2$ and $q < n$ there exists a constant $C > 0$ such that
\begin{align}
\label{Eq: Lp-Lq gradient estimates}
 \| \nabla T(t) f \|_{\L^q(\Omega ; \BC^{n^2})} \leq C t^{- \frac{1}{2} - \frac{n}{2} (\frac{1}{p} - \frac{1}{q})} \| f \|_{\L^p_{\sigma} (\Omega)} \qquad (t > 0 , f \in \L^p_{\sigma} (\Omega)).
\end{align}
\end{theo}

To state the second main result consider the Cauchy problem
\begin{align}
\label{Eq: Instationary system}
\left\{ \begin{aligned}
 \partial_t u +  A_p u &= f && \text{in } (0 , \infty) \times \Omega, \\
 u(0) &= u_0 && \text{in } \Omega.
\end{aligned} \right.
\end{align}
Given $s \in (1 , \infty)$, the Stokes operator $A_p$ is said to admit maximal $\L^s$-regularity if there exists a constant $C > 0$ such that for all $u_0$ in the real interpolation space $(\L^p_{\sigma} (\Omega) , \CD(A_p))_{1 - 1 / s , s}$ and for all $f \in \L^s (0 , \infty ; \L^p_{\sigma} (\Omega))$ the system~\eqref{Eq: Instationary system} has a unique solution $u$, that is almost everywhere differentiable in time, that satisfies $u(t) \in \CD(A_p)$ for almost every $t > 0$, and
\begin{align*}
 \| \partial_t u \|_{\L^s (0 , \infty ; \L^p_{\sigma} (\Omega))} + \| A_p u \|_{\L^s (0 , \infty ; \L^p_{\sigma} (\Omega))} \leq C \big( \| f \|_{\L^p_{\sigma} (\Omega)} + \| u_0 \|_{(\L^p_{\sigma} (\Omega) , \CD(A_p))_{1 - 1 / s , s}} \big).
\end{align*}
It is well-known, see~\cite{Cannarsa_Vespri, Denk_Hieber_Pruess}, that maximal $\L^s$-regularity is independent of $s$ and will thus be called maximal regularity. We have the following theorem.

\begin{theo}
\label{Thm: Maximal regularity}
Let $\Omega$ be an exterior Lipschitz domain in $\BR^n$ $(n\geq 3)$. Then there exists $\varepsilon > 0$ such that for all numbers $p$ that satisfy~\eqref{cond-p} the Stokes operator $A_p$ has maximal regularity.
\end{theo}

Our last main result concerns the existence of mild solutions to the three-dimensional Navier--Stokes equations
\begin{align}
\label{Eq: Navier-Stokes equations}
\left\{\begin{aligned}
\partial_t u - \Delta u + (u \cdot \nabla) u + \nabla \pi & = 0 && \text{in } (0 , T) \times \Omega, \\
\dv (u) & = 0 && \text{in } (0 , T) \times \Omega, \\
u & = 0 && \text{on } (0 , T) \times \pd\Omega, \\
u(0) &= a && \text{in } \Omega.
\end{aligned}\right.
\end{align}
Given $a \in \L^p_{\sigma} (\Omega)$, we say that a continuous function $u \in \C ([0 , T) ; \L^p_{\sigma} (\Omega))$ is a mild solution to~\eqref{Eq: Navier-Stokes equations} if $u(0) = a$ and if
\begin{align*}
 u (t) = T(t) a - \int_0^t T(t - s) \BP \dv (u(s) \otimes u(s)) \; \mathrm{d} s \qquad (0 < t < T).
\end{align*}
Relying on Theorem~\ref{th-1.1}, we obtain the following theorem.

\begin{theo}
\label{Thm: Navier-Stokes}
Let $\Omega \subset \BR^3$ be an exterior Lipschitz domain, $\varepsilon > 0$ be as in Theorem~\ref{th-1.1}, and $3 \leq r < \min\{3 + \varepsilon , 4\}$. For $a \in \L^r_{\sigma} (\Omega)$ the following statements are valid.
\begin{enumerate}
 \item \label{Item: Main 1}There exists a number $T_0 > 0$ and a mild solution $u$ to~\eqref{Eq: Navier-Stokes equations} on $[0 , T_0)$ that satisfies for all $p$ with $r \leq p < \min\{3 + \varepsilon , 4\}$
\begin{align*}
 t \mapsto t^{\frac{3}{2} (\frac{1}{r} - \frac{1}{p})} u(t) \in \B \C ([0 , T_0) ; \L^p_{\sigma} (\Omega)) \quad \text{and} \quad \| u(t) - a \|_{\L^r_{\sigma} (\Omega)} \to 0 \quad \text{as} \quad t \searrow 0.
\end{align*}
 Moreover, if $p > r$, then
\begin{align*}
 t^{\frac{3}{2} (\frac{1}{r} - \frac{1}{p})} \| u(t) \|_{\L^p_{\sigma} (\Omega)} \to 0 \quad \text{as} \quad t \searrow 0.
\end{align*}
 \item If $r > 3$, there exists a constant $C > 0$, depending only on $r$, $p$, and the constants in the estimates in Theorem~\ref{th-1.1}, such that
\begin{align*}
 T_0 \geq C \| a \|_{\L^r_{\sigma}(\Omega)}^{- \frac{2 r}{r - 3}}.
\end{align*}
 \item For all $3 \leq p < \min\{ 3 + \varepsilon , 4 \}$ there are positive constants $C_1 , C_2 > 0$, depending only on $p$ and the constants in the estimates in Theorem~\ref{th-1.1}, such that if $\| a \|_{\L^3_{\sigma} (\Omega)} < C_1$, the solution obtained in~\eqref{Item: Main 1} is global in time, {\em i.e.}, $T_0 = \infty$. Moreover, it satisfies the estimate
\begin{align*}
 \| u(t) \|_{\L^p_{\sigma} (\Omega)} \leq C_2 t^{- \frac{3}{2} (\frac{1}{3} - \frac{1}{p})} \qquad (0 < t < \infty).
\end{align*}
\end{enumerate}
\end{theo}

Concerning the uniqueness in Theorem~\ref{Thm: Navier-Stokes} we refer to the weak-strong uniqueness result of Kozono~\cite[Thm.~2]{Kozono}, which is also valid in thee dimensional exterior Lipschitz domains. \par

The proofs of Theorems~\ref{th-1.1} and~\ref{Thm: Maximal regularity} rely on the investigation of the Stokes resolvent problem~\eqref{Eq: Stokes resolvent problem}. More precisely, standard semigroup theory implies that the bounded analyticity of the Stokes semigroup follows, once it is shown that for some $\theta \in (\pi / 2 , \pi)$ the sector $\Sigma_{\theta}$ is contained in the resolvent set $\rho(- A_p)$ of $- A_p$ and once the resolvent satisfies the bound
\begin{align}
\label{Eq: Resolvent estimate in introduction}
 \| \lambda (\lambda + A_p)^{-1} f \|_{\L^p_{\sigma} (\Omega)} \leq C \| f \|_{\L^p_{\sigma} (\Omega)} \qquad (\lambda \in \Sigma_{\theta} , f \in \L^p_{\sigma} (\Omega))
\end{align}
with a constant $C > 0$ independent of $u$, $f$, and $\lambda$. To achieve this inequality for large values of $\lambda$, we follow the approach of Geissert \textit{et al}.~\cite{GHHSK} and show that the solution to the resolvent problem in an exterior domain is essentially a perturbation of the sum of solutions to a problem on the whole space and to a problem on a bounded Lipschitz domain with appropriately chosen data. A crucial point in this approach is to prove decay with respect to the resolvent parameter $\lambda$ of the $\L^p$-norm of the pressure that appears in the resolvent problem on a \textit{bounded} Lipschitz domain. The main part of this paper deals, however, with the analysis of the Stokes resolvent estimate for small values of $\lambda$. Here, a refined analysis based on Fredholm theory is performed and resembles to some extend to Iwashita's proof~\cite{Iwashita} of $\L^p$-$\L^q$-estimates of the Stokes semigroup on smooth exterior domains. The main motivation to follow our strategy to prove the resolvent bounds~\eqref{Eq: Resolvent estimate in introduction} is to obtain a proof that works literally for the maximal regularity property of $A_p$ as well. \par
To prove the maximal regularity property of $A_p$ a randomized version of the resolvent estimate~\eqref{Eq: Resolvent estimate in introduction} is required. As this is in fact equivalent to boundedness properties in vector-valued Lebesgue spaces, see Remark~\ref{Rem: Square function estimates} below, existing proofs~\cite{Borchers_Sohr, Giga_Sohr, Borchers_Miyakawa, Farwig_Sohr} of resolvent estimates that rely on a contradiction argument do not work for the maximal regularity property, due to the lack of compact embeddings for the vector-valued Lebesgue spaces under consideration. The first proof of the maximal regularity property on the finite time interval $(0 , T)$ on smooth exterior domains was given by Solonnikov~\cite{Solonnikov} and the first proof on the infinite time interval $(0 , \infty)$ was given by Giga and Sohr in the two papers~\cite{Giga_Sohr, GS}. The latter relies on the property of bounded imaginary powers and the Dore--Venni theorem and differs completely from our approach. \par
The rest of this paper is organized as follows:
In Section~\ref{sec-2} we introduce some notation and important preliminary results. In Section~\ref{Sec: Properties the Stokes operators on the whole space and on bounded Lipschitz domains}, we discuss properties of solutions to the Stokes resolvent problem on the whole space and on bounded Lipschitz domains. In Section~\ref{sec-3} we are concerned with the decay estimate of the pressure on bounded Lipschitz domains with respect to the resolvent parameter $\lambda$. In Section~\ref{sec-4} we deal with the proofs of Theorems~\ref{th-1.1},~\ref{Thm: Maximal regularity}, and~\ref{Thm: Navier-Stokes}.

\section{Preliminaries}
\label{sec-2}
\noindent
In the whole article the space dimension $n \in \BN$ satisfies $n \geq 3$. Let $\Xi \subset \BR^n$ be an open set and $1 < p < \infty$. As was already described in the introduction, we denote by $\L^p_{\sigma} (\Xi)$ and by $\W^{1 , p}_{0 , \sigma} (\Xi)$ the closure of $\C_{c , \sigma}^{\infty} (\Xi)$ in the respective norms. For a Banach space $X$, we denote by $\L^p (\Xi ; X)$ the usual Bochner--Lebesgue space and by $\C ([0 , T) ; X)$ and $\B\C([0 , T) ; X)$ the spaces of continuous and bounded and continuous functions on the interval $[0 , T)$, respectively. If the integrability conditions of functions in Sobolev spaces hold only on compact subsets of $\Xi$, then the space will be attached with the subscript $\mathrm{loc}$. For $s > 0$ and $k \in \BN$, the standard Bessel potential spaces are denoted by $\H^{s , p} (\Xi ; \BC^k)$. The H\"older conjugate exponent of $p$ is denoted by $p^{\prime}$. By $C > 0$ we will often denote a generic constant that does not depend on the quantities at stake. \par
In the following, we introduce the notion of exterior Lipschitz domains that is considered in this paper.

\begin{defi}
\label{Def: Exterior domains}
An exterior Lipschitz domain $\Omega \subset \BR^n$ is the complement of a bounded Lipschitz domain $D \subset \BR^n$, \textit{i.e.}, $\Omega := \BR^n \setminus D$. \par
A bounded Lipschitz domain $D \subset \BR^n$ is a bounded, open, and connected set that satisfies the following condition.
For each $x_0\in\pd D$, there exists
a Lipschitz function $\zeta: \BR^{n - 1}\to \BR$, a coordinate system $(x',x_n)$, and a radius $r > 0$ such that
\begin{align*}
B_r (x_0) \cap D &= \{(x', x_n) \in \BR^n \mid x_n > \zeta (x') \} \cap B_r (x_0), \\
B_r (x_0) \cap \pd D &= \{(x', x_n) \in \BR^n \mid x_n = \zeta (x') \} \cap B_r (x_0),
\end{align*}
where $B_r (x_0)$ denotes the ball with radius $r$ centered at $x_0$ and $x^{\prime} := (x_1 , \dots , x_{n - 1})$.
\end{defi}

\begin{rema}
Notice that the definition of exterior Lipschitz domains stated above excludes the presence of holes inside the exterior domain. The exact reason for this technical assumption is pinpointed to the discussion of the Helmholtz projection on exterior Lipschitz domains in Section~\ref{Sec: The Helmholtz projection on exterior Lipschitz domains}. More precisely, this assumption comes from the fact that Lang and M\'endez resolved in~\cite[Thm.~5.8]{LM} only the Neumann problem on exterior Lipschitz domains with \textit{connected} boundary. However, it should be possible to add holes by adapting the methods of~\cite{Dorina_Mitrea}. Notice that the rest of this paper works perfectly also with holes appearing in the exterior domain.
\end{rema}

\subsection{A digression on the Helmholtz projection}
\label{Sec: The Helmholtz projection on exterior Lipschitz domains}

For a domain $\Xi \subset \BR^n$ the Helmholtz projection $\BP_{2 , \Xi}$ on $\L^2 (\Xi ; \BC^d)$ is the orthogonal projection onto $\L^2_{\sigma} (\Xi)$. It is well-known that the Helmholtz projection induces the orthogonal decomposition
\begin{align*}
 \L^2 (\Xi ; \BC^n) = \L^2_{\sigma} (\Xi) \oplus \mathrm{G}_2 (\Xi),
\end{align*}
where for $1 < p < \infty$
\begin{align*}
 \mathrm{G}_p (\Xi) := \{ \nabla g \in \L^p (\Xi ; \BC^n) \mid g \in \L^p_{\mathrm{loc}} (\Xi) \}.
\end{align*}
Thus, for $1 < p < \infty$, we say that the Helmholtz decomposition of $\L^p (\Xi ; \BC^n)$ exists, if an algebraic and topological decomposition of the form
\begin{align}
\label{Eq: Helmholtz general}
 \L^p (\Xi ; \BC^n) = \L^p_{\sigma} (\Xi) \oplus \mathrm{G}_p (\Xi)
\end{align}
exists. The Helmholtz projection $\BP_{p , \Xi}$ on $\L^p (\Xi ; \BC^n)$ is then defined as the projection of $\L^p (\Xi ; \BC^n)$ onto $\L^p_{\sigma} (\Xi)$. In the case $\Xi = \BR^n$, it is well-known that the Helmholtz decomposition exists for all $1 < p < \infty$, see, \textit{e.g.}, Galdi~\cite[Thm.~III.1.2]{G}. In this case, the projection can be written as
\begin{align}
\label{Eq: Helmholtz whole space}
 \BP_{p , \BR^n} := \CF^{-1} \Big[ 1 - \frac{\xi \otimes \xi}{\lvert \xi \rvert^2} \Big] \CF,
\end{align}
where $\CF$ denotes the Fourier transform and $\CF^{-1}$ its inverse. \par
If $\Xi = D$, where $D \subset \BR^n$ denotes a bounded Lipschitz domain, then it is shown by Fabes, M\'endez, and Mitrea~\cite{FMM}, that there exists $\varepsilon = \varepsilon (D) > 0$ such that the Helmholtz decomposition of $\L^p (D ; \BC^n)$ exists if
\begin{align}
\label{Eq: p condition for Helmholtz}
 \Big\lvert \frac{1}{p} - \frac{1}{2} \Big\rvert < \frac{1}{6} + \varepsilon.
\end{align}
It is also shown in~\cite[Thm.~12.2]{FMM} that~\eqref{Eq: p condition for Helmholtz} is sharp. \par
If $\Xi = \Omega$, where $\Omega \subset \BR^n$ is an exterior Lipschitz domain \textit{with connected boundary}, it was proven by Lang and M\'endez~\cite[Thm.~6.1]{LM}, that the Helmholtz decomposition of $\L^p (\Omega ; \BC^n)$ exists for all $p$ that satisfy
\begin{align*}
 \Big\lvert \frac{1}{p} - \frac{1}{2} \Big\rvert < \min\Big\{ \frac{1}{6}  + \varepsilon , \frac{1}{2} - \frac{1}{n} \Big\}. 
\end{align*}
This is a fatal fact as in three dimensions this condition exhibits only the interval $3 / 2 < p < 3$ while it is crucial for the existence theory of the Navier--Stokes equations in the critical space $\L^{\infty} (0 , \infty ; \L^3 (\Omega))$ to have information for $p$ in the interval $[3 , 3 + \varepsilon)$, \textit{cf.},~\cite{K, Giga_Miyakawa, Giga, Tol, Tol-phd} for the cases of the whole space and bounded smooth/Lipschitz domains. \par
In the following, we review the existence proof of Lang and M\'endez and point out a slight modification in order to recover the interval~\eqref{Eq: p condition for Helmholtz} for exterior Lipschitz domains $\Omega$ with connected boundary. (In fact, we review the proof presented in~\cite{FMM} since this proof was left out by Lang and M\'endez). Notice that
\begin{align}
\label{Eq: Characterization of solenoidal spaces by normal component}
 \L^p_{\sigma} (\Omega) = \{ u \in \L^p (\Omega ; \BC^n) \mid \dv(u) = 0 \text{ in } \Omega , \nu \cdot u = 0 \text{ on } \partial \Omega \},
\end{align}
where $\nu$ denotes the outward unit vector field to $\partial \Omega$, see~\cite[Thm.~III.2.3]{G}. Notice that for $u \in \L^p (\Omega ; \BC^n)$ with $\dv(u) \in \L^p (\Omega)$ the expression $\nu \cdot u$ is regarded as an element in the Besov space $\B^{- 1 / p}_{p , p} (\partial \Omega)$ and is defined by integration by parts
\begin{align}
\label{Eq: Definition of normal trace}
 \langle \nu \cdot u , \varphi \rangle_{\B^{- 1 / p}_{p , p} , \B^{1 / p}_{p^{\prime} , p^{\prime}}} = \int_{\Omega} \dv (u) E \varphi \dx + \int_{\Omega} u \cdot \nabla E \varphi \dx \qquad (\varphi \in \B^{1 / p}_{p^{\prime} , p^{\prime}} (\partial \Omega)). 
\end{align}
In this formula, $E$ denotes an extension operator that maps $\B^{1 / p}_{p^{\prime} , p^{\prime}} (\partial \Omega)$ boundedly into $\W^{1 , p} (\Omega)$. For a construction of $E$, see~\cite[Ch.~VII]{Jonsson_Wallin}. Notice that $\nu \cdot u$ is independent of the respective extension $E \varphi$, see~\cite[Lem.~5.5]{LM}. \par
To prove the existence of the Helmholtz decomposition, let $u \in \L^2 (\Omega ; \BC^n) \cap \L^p (\Omega ; \BC^n)$. Let $\Pi_{\Omega} (\dv(u))$ denote the Newton potential of $\dv(u)$ extended by zero to $\BR^n$. Choose a function $\psi$ that solves the Neumann problem
\begin{align}
\tag{Neu}
\label{Eq: Neumann problem}
 \left\{ \begin{aligned}
 \Delta \psi &= 0 && \text{in } \Omega \\
 \nu \cdot \nabla \psi &= h && \text{on } \partial \Omega \\
 \nabla \psi &\in \L^p (\Omega ; \BC^n),
 \end{aligned} \right.
\end{align}
with $h := \nu \cdot(u - \nabla \Pi_{\Omega} (\dv (u)))$. Then one argues that the equality
\begin{align*}
 \BP_{2 , \Omega} u = u - \nabla \Pi_{\Omega} (\dv (u)) - \nabla \psi
\end{align*}
holds and it is shown by Lang and M\'endez that for $3 / 2 < p < 3$ the right-hand side gives rise to $\L^p$-boundedness estimates. These estimates in turn follow from
\begin{align*}
 \| \nabla \psi \|_{\L^p (\Omega ; \BC^n)} \leq C \| h \|_{\B^{- 1 / p}_{p , p} (\partial \Omega)} \quad \text{and} \quad \| \nabla \Pi_{\Omega} (\dv (u)) \|_{\L^p (\Omega ; \BC^n)} \leq C \| u \|_{\L^p (\Omega ; \BC^n)}.
\end{align*}
While the estimate on the left-hand side is valid for all $p$ that satisfy~\eqref{Eq: p condition for Helmholtz} by~\cite[Thm.~5.8]{LM}, the estimate on the right-hand side is valid for all $p$ that satisfy $n / (n - 1) < p < n$ by~\cite[Cor.~3.3]{LM}. \par
To get rid of the condition $n / (n - 1) < p < n$, we replace in the definition of $h$ the term $\nabla \Pi_{\Omega} \dv(u)$ by $\CF^{-1} \xi \otimes \xi \lvert \xi \rvert^{- 2} \CF U$, where $U$ denotes the extension of $u$ to $\BR^n$ by zero. Thus, by virtue of~\eqref{Eq: Helmholtz whole space}, let $h$ be given by $h := \nu \cdot (\BP_{p , \BR^n} U)|_{\partial \Omega}$. Since the divergence of $\BP_{p , \BR^n} U$ is clearly an $\L^p$-function, $h$ is well-defined as an element in $\B^{- 1 / p}_{p , p} (\partial \Omega)$ and its norm is estimated by virtue of~\eqref{Eq: Definition of normal trace} as
\begin{align*}
 \| h \|_{\B^{- 1 / p}_{p , p} (\partial \Omega)} \leq \sup_{\varphi} \Big\lvert \int_{\Omega} \BP_{p , \BR^n} U \cdot \nabla E \varphi \dx \Big\rvert \leq C \| u \|_{\L^p (\Omega ; \BC^n)}.
\end{align*}
Here the supremum is taken over all functions $\varphi \in \B^{1 / p}_{p^{\prime} , p^{\prime}} (\partial \Omega)$ with norm less or equal to one and $C > 0$ is the product of the boundedness constants of the Helmholtz projection $\BP_{p , \BR^n}$ and of the extension operator $E : \B^{1 / p}_{p^{\prime} , p^{\prime}} (\partial \Omega) \to \W^{1 , p^{\prime}} (\Omega)$. \par
By virtue of~\cite[Thm.~5.8]{LM} there exists a unique function $\psi$ that satisfies~\eqref{Eq: Neumann problem} whenever $p$ satisfies~\eqref{Eq: p condition for Helmholtz}. Now, we show the equality
\begin{align}
\label{Eq: Represenation of Helmholtz by Neumann}
 \BP_{2 , \Omega} u = (\BP_{p , \BR^n} U)|_{\Omega} - \nabla \psi.
\end{align}
Since by assumption $u$ lies in $\L^2 (\Omega ; \BC^n) \cap \L^p (\Omega ; \BC^n)$ it holds $\BP_{p , \BR^n} U = \BP_{2 , \BR^n} U$. Furthermore, if $\nabla g \in \mathrm{G}_2 (\BR^n)$ with $U = \BP_{2 , \BR^n} U + \nabla g$, then
\begin{align*}
 u = (\BP_{p , \BR^n} U)|_{\Omega} + \nabla g|_{\Omega} \qquad \Leftrightarrow \qquad u = (\BP_{p , \BR^n} U)|_{\Omega} - \nabla \psi + \nabla (g|_{\Omega} + \psi).
\end{align*}
Notice that $\nabla (g|_{\Omega} + \psi) \in \mathrm{G}_2 (\Omega)$ and that $(\BP_{p , \BR^n} U)|_{\Omega} - \nabla \psi$ is divergence free. Moreover, its normal trace vanishes by the construction of $\psi$. By virtue of~\eqref{Eq: Characterization of solenoidal spaces by normal component} it follows that $(\BP_{p , \BR^n} U)|_{\Omega} - \nabla \psi \in \L^2_{\sigma} (\Omega)$ and by the uniqueness of the Helmholtz decomposition, it follows the equality in~\eqref{Eq: Represenation of Helmholtz by Neumann}. \par
This proves already the existence of the Helmholtz decomposition on $\L^p (\Omega ; \BC^n)$ for all $p$ subject to~\eqref{Eq: p condition for Helmholtz}, but only if $\partial \Omega$ is \textit{connected}. \par
To use this result to obtain the Helmholtz decomposition for exterior domains subject to Definition~\ref{Def: Exterior domains} proceed as follows. Decompose $\Omega$ into its connected components
\begin{align*}
 \Omega = \Omega_0 \cup \bigcup_{k = 1}^N \Omega_k,
\end{align*}
where $\Omega_0$ is the unbounded connected component and $\Omega_k$ $(k = 1 , \dots , N)$ are bounded Lipschitz domains. Thus, by~\cite[Thm.~11.1]{FMM}, there exists $\varepsilon > 0$ such that the Helmholtz projection $\BP_{p , \Omega_k}$ is bounded on $\L^p (\Omega_k ; \BC^n)$ for all $p$ subject to~\eqref{Eq: p condition for Helmholtz} and $k = 1 , \dots , N$. To define the Helmholtz projection on $\Omega$, define
\begin{align*}
 [\BP_{p , \Omega} f] (x) := [\BP_{p , \Omega_k} R_k f] (x) \qquad (x \in \Omega_k , k = 0 , \dots , N , f \in \L^p (\Omega ; \BC^n)),
\end{align*}
where $R_k$ denotes the restriction operator of functions on $\Omega$ to $\Omega_k$. \par
The discussion of this section leads to the following proposition on the existence of the Helmholtz decomposition on exterior Lipschitz domains.

\begin{prop}
\label{Prop: Helmholtz on exterior}
Let $\Omega \subset \BR^n$ be an exterior Lipschitz domain and $p$ be subject to~\eqref{Eq: p condition for Helmholtz}. Then the Helmholtz decomposition~\eqref{Eq: Helmholtz general} exists.
\end{prop}

\subsection{Maximal regularity}

Recall the definition of maximal regularity below~\eqref{Eq: Instationary system} and recall that given $\theta \in (0 , \pi)$ the sector $\Sigma_{\theta}$ is given by $\Sigma_{\theta} := \{ z \in \BC \setminus \{ 0 \} \mid \lvert \arg (z) \rvert < \theta \}$. Clearly, the definition of maximal regularity can be generalized to a closed operator $\CA : \CD(\CA) \subset X \to X$ on a Banach space $X$ such that $- \CA$ generates a bounded analytic semigroup on $X$. In this case, there exists an angle $\theta \in (\pi / 2 , \pi)$ such that $\Sigma_{\theta} \subset \rho(- \CA)$. The following characterization is due to Weis~\cite[Thm.~4.2]{Weis}.

\begin{prop}
\label{Prop: Weis}
Let $X$ be a space of type $\mathrm{UMD}$ and let $- \CA$ be the generator of a bounded analytic semigroup on $X$. Then $\CA$ has maximal regularity if and only if there exists $\theta \in (\pi / 2 , \pi)$ such that $\{ \lambda (\lambda + \CA)^{-1} \}_{\lambda \in \Sigma_{\theta}}$ is $\CR$-bounded in $\CL(X)$.
\end{prop}

It is well-known that $\L^p$-spaces for $1 < p < \infty$ are of type $\mathrm{UMD}$. Moreover, all closed subspaces of $\mathrm{UMD}$-spaces are of type $\mathrm{UMD}$. As a consequence, $\L^p_{\sigma} (\Xi)$ is a $\mathrm{UMD}$-space for all measurable sets $\Xi \subset \BR^n$. The definition of $\CR$-bounded families of operators reads as follows.

\begin{defi}
Let $X$ and $Y$ be Banach spaces.
A family of operators $\CT \subset \CL (X, Y)$ is said to be \textit{$\CR$-bounded} if there exists a positive constant $C > 0$ such that for any $N\in \BN$, $T_j \in \CT$, $x_j \in X$ $(j = 1, \dots, N)$ the inequality
\begin{align}
\label{2.1}
\Big\| \sum_{j = 1}^{N} r_j (\cdot) T_j x_j \Big\|_{\L^2 (0 , 1 ; Y)}
\leq C \Big\| \sum_{j = 1}^{N} r_j (\cdot) x_j \Big\|_{\L^2 (0 , 1 ; X)}
\end{align}
holds. Here, $r_j (t) := \mathrm{sgn} (\sin(2^j \pi t))$ are the \textit{Rademacher-functions}. The infimum over all constants $C > 0$ such that~\eqref{2.1} holds is said to be the $\CR$-bound of $\CT$ and will be denoted by $\CR_{X \to Y} \{\CT\}$.
If $X = Y$ we simply write $\CR_{X} \{\CT\}$.
\end{defi}

\begin{rema}
\label{Rem: R-boundedness implies uniform boundedness}
Notice that $\CR$-boundedness of a family of operators implies its uniform boundedness. If $X$ and $Y$ are Hilbert spaces, then $\CR$-boundedness is equivalent to uniform boundedness~\cite[Rem.~3.2]{Denk_Hieber_Pruess}.
\end{rema}

\begin{rema}
\label{Rem: Square function estimates}
Let $1 < p , q < \infty$, $k , m \in \BN$, and $\Xi \subset \BR^n$ be measurable. It is well-known, see~\cite[Rem.~3.2]{Denk_Hieber_Pruess}, that there exists a constant $C > 0$ such that for all $N \in \BN$ and $f_j \in \L^p (\Xi ; \BC^k)$ it holds
\begin{align}
\label{Eq: Equivalence R-sum and square function term}
 C^{-1} \Big\| \sum_{j = 1}^N r_j (\cdot) f_j \Big\|_{\L^2 (0 , 1 ; \L^p (\Xi ; \BC^k))} \leq \Big\| \Big[ \sum_{j = 1}^N \lvert f_j \rvert^2 \Big]^{1 / 2} \Big\|_{\L^p (\Xi)} \leq C \Big\| \sum_{j = 1}^N r_j (\cdot) f_j \Big\|_{\L^2 (0 , 1 ; \L^p (\Xi ; \BC^k))}.
\end{align}
Thus, $\CR$-boundedness of an operator family $\CT \subset \CL (\L^p (\Xi ; \BC^k) , \L^q (\Xi ; \BC^m))$ is equivalent to the validity of square function estimates of the following form. Namely, there exists a constant $C > 0$ such that for all $N \in \BN$, $T_j \in \CT$, and $f_j \in \L^p (\Xi ; \BC^k)$ it holds
\begin{align*}
 \Big\| \Big[ \sum_{j = 1}^N \lvert T_j f_j \rvert^2 \Big]^{1 / 2} \Big\|_{\L^q (\Xi)} \leq C \Big\| \Big[ \sum_{j = 1}^N \lvert f_j \rvert^2 \Big]^{1 / 2} \Big\|_{\L^p (\Xi)}.
\end{align*}
Notice further that this is equivalent to boundedness of the family
\begin{align*}
 \CS := \{ (T_1 , \dots , T_N , 0 , \dots) \mid N \in \BN , T_j \in \CT \text{ for } 1 \leq j \leq N \} \subset \CL(\L^p (\Xi ; \ell^2(\BC^k)) , \L^q (\Xi ; \ell^2(\BC^m))).
\end{align*}
Here, $\ell^2(\BC^l)$ denotes the space of square summable sequences that take values in $\BC^l$ for $l \in \BN$. Moreover, $(T_1 , \dots , T_N , 0 , \dots)$ acts componentwise on an element $f = (f_j)_{j \in \BN} \in \L^p (\Xi ; \ell^2(\BC^k))$.
\end{rema}

For further reference, we record the contraction principle of Kahane, see, \textit{e.g.},~\cite[Lem.~3.5]{Denk_Hieber_Pruess}.

\begin{prop}
\label{Prop: Kahane}
Let $X$ be a Banach space, $N \in \BN$, $x_j \in X$, and $\alpha_j , \beta_j \in \BC$ such that $\lvert \alpha_j \rvert \leq \lvert \beta_j \rvert$ $(j = 1 , \dots , N)$. Then
\begin{align*}
 \Big\| \sum_{j = 1}^N r_j (\cdot) \alpha_j x_j \Big\|_{\L^2 (0 , 1 ; X)} \leq 2 \Big\| \sum_{j = 1}^N r_j (\cdot) \beta_j x_j \Big\|_{\L^2 (0 , 1 ; X)}.
\end{align*}
\end{prop}

\section{Properties of the Stokes operators on the whole space and on bounded Lipschitz domains}
\label{Sec: Properties the Stokes operators on the whole space and on bounded Lipschitz domains}

\noindent In this section we are going to present important properties of the Stokes resolvent problems on the whole space and on bounded Lipschitz domains.

\subsection{The Stokes operator on the whole space}

The Stokes operator on $\L^p_\sigma (\BR^n)$ is defined as
\begin{align*}
A_{p, \BR^n} u := - \BP_{p , \BR^n} \Delta_{p , \BR^n} u \quad \text{for } u \in \CD(A_{p , \BR^n}) := \W^{2 , p} (\BR^n ; \BC^n) \cap \L^p_{\sigma} (\BR^n).
\end{align*}
By the very definition it is clear that $\BP_{p , \BR^n}$ and $\Delta_{p , \BR^n}$ commute so that $A_{p , \BR^n} u = - \Delta_{p , \BR^n} u$ is valid for $u \in \CD(A_{p , \BR^n})$. Let $\theta \in (0 , \pi)$. Then the resolvent problem for the Stokes operator with general right-hand side $f \in \L^p (\BR^n ; \BC^n)$
\begin{align}
\label{Eq: Resolvent problem on the whole space}
\left\{ \begin{aligned}
 \lambda u - \Delta u + \nabla \pi &= f && \text{in } \BR^n \\
 \dv(u) &= 0 && \text{in } \BR^n
\end{aligned} \right.
\end{align}
can be solved as follows: First, decompose $f$ by means of~\eqref{Eq: Helmholtz general} as $f = \BP_{p , \BR^n} f + \nabla g$. Then the solutions to the resolvent problem are given by $u := (\lambda - \Delta_{p , \BR^n})^{-1} \BP_{p , \BR^n} f$ and $\pi := g$. By virtue of~\cite[Sec.~3]{Shibata_Shimizu} we have the following proposition.

\begin{prop}\label{prop-2.5}
Let $1 < p < \infty$ and $\theta \in (0 , \pi)$. Then for all $f \in \L^p (\BR^n ; \BC^n)$ and all $\lambda \in \Sigma_{\theta}$ the resolvent problem~\eqref{Eq: Resolvent problem on the whole space} has a unique solution $(u , \pi)$ in $\CD(A_{p , \BR^n}) \times \mathrm{G}_p (\BR^n)$ (with $\pi$ being unique up to an additional constant). The function $u$ is given by $u = (\lambda - \Delta_{p , \BR^n})^{-1} \BP_{p , \BR^n} f$ and $\pi$ is given by $\pi = g$, where $g$ satisfies $f = \BP_{p , \BR^n} f + \nabla g$. Moreover, there exists a constant $C > 0$ such that
\begin{align*}
 \CR_{\L^p (\BR^n ; \BC^n) \to \L^p_{\sigma} (\BR^n)} \big\{ \lambda (\lambda + A_{p , \BR^n})^{-1} \BP_{p , \BR^n} \mid \lambda \in \Sigma_{\theta}\big\} &\leq C \\
 \CR_{\L^p (\BR^n ; \BC^n) \to \L^p (\BR^n ; \BC^{n^2})} \big\{ \lvert \lambda \rvert^{1 / 2} \nabla (\lambda + A_{p , \BR^n})^{-1} \BP_{p , \BR^n} \mid \lambda \in \Sigma_{\theta}\big\} &\leq C \\
 \CR_{\L^p (\BR^n ; \BC^n) \to \L^p (\BR^n ; \BC^{n^3})} \big\{ \nabla^2 (\lambda + A_{p , \BR^n})^{-1} \BP_{p , \BR^n} \mid \lambda \in \Sigma_{\theta}\big\} &\leq C.
\end{align*}
\end{prop}



\subsection{The Stokes operator on bounded domains}

Let $D \subset \BR^n$ be a bounded Lipschitz domain. As in~\cite{She}, we define for $1 < p < \infty$ the Stokes operator $A_{p , D}$ to be
\begin{align*}
 A_{p , D} u := - \Delta u + \nabla \pi, \quad \text{for } u \in \CD(A_{p , D}) := \{ u \in \W^{1 , p}_{0 , \sigma} (D) \mid \exists \pi \in \L^p (D) \text{ with } - \Delta u + \nabla \pi \in \L^p_{\sigma} (D) \}.
\end{align*}
Here, the relation $- \Delta u + \nabla \pi \in \L^p_{\sigma} (D)$ is understood in the sense of distributions. Notice that due to the Lipschitz boundary one can in general not expect that $u \in \W^{2 , p} (D ; \BC^n)$ and $\pi \in \W^{1 , p} (D)$ holds for $u \in \CD(A_{p , D})$. However, $u \in \W^{2 , p}_{\loc} (D ; \BC^n)$ and $\pi \in \W^{1 , p}_{\loc} (D)$ holds by inner regularity~\cite[Thm.~IV.4.1]{G}. We summarize useful properties, extending the seminal paper of Shen~\cite{She}. These statements can be found in~\cite[Prop.~13]{KW},~\cite[Thm.~5.2.24]{Tol-phd}, and~\cite[Thm.~1.1]{Tol}.

\begin{prop}
\label{prop-3.2}
Let $D$ be a bounded Lipschitz domain in $\BR^n$ and $\theta \in (0 , \pi)$. Then there exists a positive constant $\varepsilon > 0$ depending only on $n$, $\theta$, and the Lipschitz geometry of $D$ such that for all $p \in (1, \infty)$ satisfying
\begin{align}
\label{Eq: Condition on p on bounded domains}
 \Big\lvert \frac{1}{p} - \frac{1}{2} \Big\rvert < \frac{1}{2n} + \varepsilon,
\end{align}
it holds $\Sigma_{\theta} \subset \rho( - A_{p , D})$ and there exists a constant $C > 0$ such that
\begin{align*}
 \CR_{\L^p (D ; \BC^n) \to \L^p_{\sigma} (D)} \big\{\lambda (\lambda + A_{p , D})^{-1} \BP_{p , D} \mid \lambda \in \Sigma_{\theta}\big\} \leq C.
\end{align*}
Moreover, for all these $p$ it holds $\CD (A^{1/2}_{p, D}) = \W^{1,p}_{0,\sigma} (D)$ and there exists a constant $C > 0$ such that
\begin{align}
\label{Eq: Square root estimate}
 \| \nabla u \|_{\L^p (D ; \BC^{n^2})} \leq C \| A^{1/2}_{p, D} u \|_{\L^p_{\sigma} (D)} \qquad (u \in \CD(A_{p , D}^{1 / 2})).
\end{align}
\end{prop}

A direct consequence of this proposition is the following lemma.

\begin{lemm}
\label{lem-3.7}
Let $D$ be a bounded Lipschitz domain in $\BR^n$ and let $p \in (1, \infty)$ satisfy~\eqref{Eq: Condition on p on bounded domains}.
Then the following estimates hold: For all $\theta \in (0, \pi)$, $\alpha \in (0, 1)$, and $\beta \in [0, 1/2]$ there exists $C > 0$ such that
\begin{align}
\label{3.4}
\CR_{\L^p (D ; \BC^n) \to \L^p_\sigma (D)} \big\{ \lvert \lambda \rvert^\alpha A_{p, D}^{1-\alpha}(\lambda + A_{p, D})^{-1} \BP_{p , D} \mid \lambda \in \Sigma_\theta \big\} \leq C,\\
\label{3.5}
\CR_{\L^p (D ; \BC^n) \to \L^p (D ; \BC^{n^2})}\big\{ \lvert \lambda \rvert^\beta \nabla (\lambda + A_{p, D})^{-1} \BP_{p , D} \mid \lambda \in \Sigma_\theta \big\} \leq C.
\end{align}
\end{lemm}

\begin{proof}
First of all, notice that~\eqref{3.4} follows by combining Proposition~\ref{prop-3.2} with~\cite[Ex.~10.3.5]{HNVW}. To prove~\eqref{3.5} let $N \in \BN$, $\lambda_j \in \Sigma_\theta$, and $f_j \in \L^p (D ; \BC^n)$ $(1 \leq j \leq N)$. Applying estimate~\eqref{Eq: Square root estimate} to the function $u := \sum_{j = 1}^N r_j (t) \lvert \lambda_j\rvert^{\beta} (\lambda_j + A_{p , D})^{-1} \BP_{p , D} f_j$ for $0 < t < 1$ followed by the boundedness of $A_{p , D}^{\beta - 1 / 2}$ and~\eqref{3.4} with $\alpha = \beta \in [0 , 1 / 2]$ delivers
\begin{align*}
 &\Big\| \sum_{j = 1}^{N} r_j (\cdot) \lvert \lambda_j \rvert^\beta \nabla (\lambda_j + A_{p, D})^{-1} \BP_{p , D}f_j \Big\|_{\L^2(0 , 1 ; \L^p (D ; \BC^{n^2}))} \\
 &\quad\leq C \Big\| \sum_{j = 1}^{N} r_j (\cdot) \lvert \lambda_j \rvert^\beta A_{p, D}^{1/2} (\lambda_j + A_{p, D})^{-1} \BP_{p , D} f_j \Big\|_{\L^2(0 , 1 ; \L^p_{\sigma} (D))} \leq C \Big\| \sum_{j = 1}^{N} r_j (\cdot) f_j \Big\|_{\L^2(0 , 1 ; \L^p (D ; \BC^n))}. \qedhere
\end{align*}
\end{proof}

\subsection{$\CR$-bounded $\L^p$-$\L^q$-estimates of the Stokes resolvent} In this section, we are going to derive the validity of $\CR$-bounded $\L^p$-$\L^q$-estimates for the Stokes resolvent on the whole space and on bounded Lipschitz domains. To this end, we employ the following abstract version of Stein's interpolation theorem, which is due to Voigt~\cite{Voigt}.

\begin{prop}
\label{Prop: Stein interpolation}
Let $(X_0 , X_1)$ and $(Y_0 , Y_1)$ be interpolation couples, let $\CX$ be dense in $X_0 \cap X_1$ with respect to the intersection space norm, and let $S := \{ w \in \BC \mid 0 \leq \Re(w) \leq 1 \}$. If $(T(z))_{z \in S}$ is a family of linear mappings $T(z) : \CX \to Y_0 + Y_1$ with the following properties:
\begin{enumerate}
 \item For all $x \in \CX$ the function $T(\cdot) x : S \to Y_0 + Y_1$ is continuous, bounded, and analytic on the interior of $S$;
 \item\label{Item: Good bounds} for $j = 0 , 1$ and $x \in \CX$ the function $\BR \ni s \mapsto T(j + \mathrm{i} s)x \in Y_j$ is continuous and 
\begin{align*}
 M_j := \sup\{ \| T(j + \mathrm{i} s) x \|_{Y_j} \mid s \in \BR , x \in \CX , \| x \|_{X_j} \leq 1 \} < \infty.
\end{align*}
\end{enumerate}
Then for all $\theta \in [0 , 1]$ it holds $T(\theta) \CX \subset [Y_0 , Y_1]_{\theta}$ and
\begin{align*}
 \| T(\theta) x \|_{[Y_0 , Y_1]_{\theta}} \leq M_0^{1 - \theta} M_1^{\theta} \| x \|_{[X_0 , X_1]_{\theta}} \qquad (x \in \CX).
\end{align*}
Here, $[X_0 , X_1]_{\theta}$ and $[Y_0 , Y_1]_{\theta}$ denote interpolation spaces with respect to the complex interpolation functor.
\end{prop}

\begin{lemm}
\label{Lem: R-bounded L^p-L^q estimates whole space}
Let $\theta \in (0 , \pi)$. For all $1 < p \leq q < \infty$ with $\sigma := n (1 / p - 1 / q) / 2 \leq 1$ there exists a constant $C > 0$ such that
\begin{align*}
 \CR_{\L^p (\BR^n ; \BC^n) \to \L^q (\BR^n ; \BC^n)} \big\{ \lvert \lambda \rvert^{1 - \sigma} (\lambda + A_{p , \BR^n})^{-1} \BP_{p , \BR^n} \mid \lambda \in \Sigma_{\theta} \big\} \leq C.
\end{align*}
\end{lemm}

\begin{proof}
Let $N \in \BN$, fix $\lambda_j \in \Sigma_{\theta}$ $(1 \leq j \leq N)$, and let $1 < p < n / 2$ and $1 / p - 1 / r = 2 / n$. Define
\begin{align*}
 X_0 = X_1 = Y_1 = \L^p (\BR^n ; \ell^2 (\BC^n)) \qquad \text{and} \qquad Y_0 = \L^r (\BR^n ; \ell^2 (\BC^n)).
\end{align*}
Define for $z \in S = \{ w \in \BC \mid 0 \leq \Re(w) \leq 1 \}$
\begin{align*}
 T (z) : X_0 &\to Y_0 +  Y_1, \\
 (f_j)_{j \in \BN} &\mapsto (T_j (z) f_j)_{j \in \BN} := \big( \lvert\lambda_1 \rvert^z (\lambda_1 + A_{p , \BR^n})^{-1} \BP_{p , \BR^n} f_1 , \dots , \lvert \lambda_N \rvert^z (\lambda_N + A_{p , \BR^n})^{-1} \BP_{p , \BR^n} f_N , 0 , \dots \big).
\end{align*}
Clearly, for each $f \in \L^p (\BR^n ; \ell^2 (\BC^n))$ the function $z \mapsto T(z) f$ is continuous and bounded on $S$ and analytic on its interior. \par
To calculate $M_0$ and $M_1$ in Proposition~\ref{Prop: Stein interpolation}~\eqref{Item: Good bounds} let $f = (f_j)_{j \in \BN}$ with $\| f \|_{\L^p (\BR^n ; \ell^2 (\BC^n))} \leq 1$ and $s \in \BR$. Notice that $\lvert \lvert \lambda_j \rvert^z \rvert = \lvert \lambda_j \rvert^{\Re(z)}$. Thus, by virtue of~\eqref{Eq: Equivalence R-sum and square function term} and Proposition~\ref{prop-2.5} there exists a constant $C > 0$ such that
\begin{align*}
 \Big\| \Big[ \sum_{j = 1}^N \lvert T_j(1 + \mathrm{i} s) f_j \rvert^2 \Big]^{1 / 2} \Big\|_{\L^p (\BR^n)} = \Big\| \Big[ \sum_{j = 1}^N \lvert T_j(1) f_j \rvert^2 \Big]^{1 / 2} \Big\|_{\L^p (\BR^n)} \leq C.
\end{align*}
Taking the supremum over $s$ and $f$ delivers $M_1 \leq C$. Notice that $C$ is uniform in $N$ and $\lambda_j$. \par
To bound $M_0$, use~\eqref{Eq: Equivalence R-sum and square function term} and Sobolev's embedding theorem to deduce
\begin{align*}
 \Big\| \Big[ \sum_{j = 1}^N \lvert T_j(\mathrm{i} s) f_j \rvert^2 \Big]^{1 / 2} \Big\|_{\L^r (\BR^n)} \leq C \Big\| \sum_{j = 1}^N \nabla^2 (\lambda_j + A_{p , \BR^n})^{-1} \BP_{p , \BR^n} r_j (\cdot) f_j \Big\|_{\L^2(0 , 1 ; \L^p (\BR^n ; \BC^n))}.
\end{align*}
By virtue of Proposition~\ref{prop-2.5}, the term on the right-hand side can again be bounded by a constant $C > 0$ that is uniform in $N$ and $\lambda_j$. It follows again $M_0 \leq C$. \par
Let $p < q < r$. Then, due to the choice of $r$, it holds for some $\theta \in (0 , 1)$
\begin{align*}
 \frac{1}{q} = \frac{\theta}{p} + \frac{1 - \theta}{r} \qquad \Leftrightarrow \qquad \theta = 1 - \frac{n}{2} \Big( \frac{1}{p} - \frac{1}{q} \Big) = 1 - \sigma.
\end{align*}
Proposition~\ref{Prop: Stein interpolation} implies the existence of a constant $C > 0$ that is uniform in $N$ and $\lambda_j$ such that for all $f = (f_j)_{j \in \BN} \in \L^p (\BR^n ; \ell^2(\BC^n))$ it holds
\begin{align*}
 \Big\| \Big[ \sum_{j = 1}^N \lvert T_j(\theta) f_j \rvert^2 \Big]^{1 / 2} \Big\|_{\L^q (\BR^n)} \leq C \| f \|_{\L^p (\BR^n ; \ell^2(\BC^n))},
\end{align*}
which is the statement of the lemma. \par
The general case $1 < p \leq q < \infty$ with $1 / p - 1 / q \leq 2 / n$ follows from the Stein interpolation theorem as well. Notice that the case $1 / p - 1 / q = 2 / n$ was already covered in the first part of the proof. Thus, let $1 / p - 1 / q < 2 / n$ and choose $1 < r < n / 2$ with $1 / r - 1 / q \leq 2 / n$. Define
\begin{align*}
 X_0 := \L^r (\BR^n ; \ell^2 (\BC^n)), \qquad X_1 = Y_0 = Y_1 := \L^q (\BR^n ; \ell^2 (\BC^n)), \qquad \text{and} \qquad \CX := X_0 \cap X_1.
\end{align*}
Moreover, define for $z \in S$
\begin{align*}
 U (z) : \CX &\to Y_0, \\
 (f_j)_{j \in \BN} &\mapsto \big( \lvert\lambda_1 \rvert^{(1 - z) \nu + z} (\lambda_1 + A_{p , \BR^n})^{-1} \BP_{p , \BR^n} f_1 , \dots , \lvert \lambda_N \rvert^{(1 - z) \nu + z} (\lambda_N + A_{p , \BR^n})^{-1} \BP_{p , \BR^n} f_N , 0 , \dots \big),
\end{align*}
where $\nu := 1 - n (1 / r - 1 / q) / 2$. Stein's interpolation theorem can now be applied in this situation as well while the uniform estimates on $M_0$ and $M_1$ follow by the result of the first part of the proof and Proposition~\ref{prop-2.5}. The proof is completed by the density of $\CX$ in $\L^p (\BR^n ; \ell^2 (\BC^n))$.
\end{proof}

An analogous result holds for the Stokes operator on bounded Lipschitz domains.

\begin{lemm}
\label{Lem: R-bounded L^p-L^q estimates Lipschitz domain}
Let $D \subset \BR^n$ be a bounded Lipschitz domain and $\theta \in (0 , \pi)$. There exists $\varepsilon > 0$ such that for all $1 < p \leq q < \infty$ with $\sigma := n (1 / p - 1 / q) / 2 \leq 1 / 2$ that both satisfy~\eqref{Eq: Condition on p on bounded domains} there exists a constant $C > 0$ such that
\begin{align*}
 \CR_{\L^p (D ; \BC^n) \to \L^q (D ; \BC^n)} \big\{ \lvert \lambda \rvert^{1 - \sigma} (\lambda + A_{p , D})^{-1} \BP_{p , D} \mid \lambda \in \Sigma_{\theta} \big\} \leq C.
\end{align*}
\end{lemm}

\begin{proof}
Let first $p$ additionally satisfy $p < n$ and let $p$ be such that there exists $r$ satisfying
\begin{align*}
 \Big\lvert \frac{1}{r} - \frac{1}{2} \Big\rvert < \frac{1}{2 n} + \varepsilon \qquad \text{and} \qquad \frac{1}{p} - \frac{1}{r} = \frac{1}{n}.
\end{align*}
Notice that such a choice is always possible. Let $N \in \BN$ and fix $\lambda_j \in \Sigma_{\theta}$ $(1 \leq j \leq N)$. Define
\begin{align*}
 X_0 = X_1 = Y_1 = \L^p (D ; \ell^2 (\BC^n)) \qquad \text{and} \qquad Y_0 = \L^r (D ; \ell^2 (\BC^n)).
\end{align*}
Define for $z \in S = \{ w \in \BC \mid 0 \leq \Re(w) \leq 1 \}$
\begin{align*}
 T (z) : X_0 &\to Y_0 +  Y_1, \\
 (f_j)_{j \in \BN} &\mapsto (T_j (z) f_j)_{j \in \BN} := \big( \lvert\lambda_1 \rvert^{\frac{1 + z}{2}} (\lambda_1 + A_{p , D})^{-1} \BP_{p , D} f_1 , \dots , \lvert \lambda_N \rvert^{\frac{1 + z}{2}} (\lambda_N + A_{p , D})^{-1} \BP_{p , D} f_N , 0 , \dots \big).
\end{align*}
To bound $M_0$ in Proposition~\ref{Prop: Stein interpolation}~\eqref{Item: Good bounds}, let $f = (f_j)_{j \in \BN} \in \L^p (D ; \ell^2 (\BC^n))$ with $\| f \|_{\L^p (D ; \ell^2 (\BC^n))} \leq 1$ and let $s \in \BR$. Use~\eqref{Eq: Equivalence R-sum and square function term} and Sobolev's embedding theorem to deduce
\begin{align*}
 \Big\| \Big[ \sum_{j = 1}^N \lvert T_j(\mathrm{i} s) f_j \rvert^2 \Big]^{1 / 2} \Big\|_{\L^r (D)} \leq C \Big\| \sum_{j = 1}^N \nabla \lvert \lambda_j \rvert^{1 / 2} (\lambda_j + A_{p , D})^{-1} \BP_{p , D} r_j (\cdot) f_j \Big\|_{\L^2(0 , 1 ; \L^p (D ; \BC^{n^2}))}.
\end{align*}
By virtue of~\eqref{3.5}, the right-hand side is bounded by a constant $C > 0$ that is uniform in $N$ and $\lambda_j$. Taking the supremum over $s$ and $f$ delivers $M_0 \leq C$. All other estimates follow now literally as in the proof of Lemma~\ref{Lem: R-bounded L^p-L^q estimates whole space} but rely on Proposition~\ref{prop-3.2} instead of Proposition~\ref{prop-2.5}.
\end{proof}

\subsection{Transference of $\L^p$-$\L^q$-estimates}

For further reference, we record the following proposition, which allows for a one-to-one correspondence between $\L^p$-$\L^q$-estimates for the semigroup and for the resolvent.

\begin{prop}
\label{Prop: Transference}
Let $- \CA$ be the generator of a bounded analytic semigroup $(S(z))_{z \in \Sigma_{\theta - \pi / 2} \cup \{ 0 \}}$ for some $\theta \in (\pi / 2 , \pi)$ on a Banach space $X$. Let $\CX \subset X$ and let $Y$ be another Banach space with $X \cap Y \neq \emptyset$. Moreover, let $B$ be a closed operator on $Y$ with domain $\CD(B)$ and let $0 \leq \alpha < 1$. Then, the following are equivalent:
\begin{enumerate}
 \item For all $x \in \CX$ and $z \in \Sigma_{\theta - \pi / 2}$ it holds $S(z) x \in \CD(B)$ and for all $\pi / 2 < \phi < \theta$ there exists $C > 0$ such that 
\begin{align}
\label{Eq: Abstract semigroup estimate}
 \| B S(z) x \|_Y \leq C \lvert z \rvert^{- \alpha} \| x \|_X \qquad (x \in \CX , z \in \Sigma_{\phi - \pi / 2}).
\end{align}
 \item For all $x \in \CX$ and $\lambda \in \Sigma_{\theta}$ it holds $(\lambda + \CA)^{-1} x \in \CD(B)$ and for all $\pi / 2 < \phi < \theta$  there exists $C > 0$ such that
\begin{align}
\label{Eq: Abstract resolvent estimate}
 \| B (\lambda + \CA)^{-1} x \|_Y \leq C \lvert \lambda \rvert^{\alpha - 1} \| x \|_X \qquad (x \in \CX , \lambda \in \Sigma_{\phi}).
\end{align}
\end{enumerate}
\end{prop}

\begin{proof}
'$(1) \Rightarrow (2)$': Notice that $- \CA$ generates a bounded analytic semigroup on $\Sigma_{\theta - \pi / 2}$ if and only if for each $\pi / 2 < \vartheta < \theta$ the operator $- \mathrm{e}^{\pm \mathrm{i} (\vartheta - \pi / 2)} \CA$ generates a bounded $C_0$-semigroup $(S_{\pm \vartheta} (t))_{t \geq 0}$ on $X$. Further, notice that 
\begin{align*}
 S_{\pm \vartheta} (t) = S (\mathrm{e}^{\pm \mathrm{i} (\vartheta - \pi / 2)} t) \qquad (t > 0).
\end{align*}
Let $\pi / 2 < \phi < \vartheta < \theta$. Standard semigroup theory implies that the resolvent is represented via the Laplace-transform of the semigroup, \textit{i.e.},
\begin{align*}
 \big(\lambda + \mathrm{e}^{\pm \mathrm{i} (\vartheta - \pi / 2)} \CA \big)^{-1} x = \int_0^{\infty} \mathrm{e}^{- \lambda t} S (\mathrm{e}^{\pm \mathrm{i} (\vartheta - \pi / 2)} t) x \dt \qquad (x \in X , \lambda \in \BC \text{ with } \Re(\lambda) > 0).
\end{align*}
The estimate on the semigroup then implies for $x \in \CX$
\begin{align*}
 \int_0^{\infty} \| \mathrm{e}^{- \lambda t} B S (\mathrm{e}^{\pm \mathrm{i} (\vartheta - \pi / 2)} t) x \|_Y \dt \leq C \int_0^{\infty} \mathrm{e}^{- \Re(\lambda) t} t^{- \alpha} \dt \| x \|_X = C^{\prime} \Re(\lambda)^{\alpha - 1} \| x \|_X.
\end{align*}
This implies that $(\mathrm{e}^{\mp \mathrm{i} (\vartheta - \pi / 2)} \lambda + \CA)^{-1} x \in \CD(B)$ and~\eqref{Eq: Abstract resolvent estimate} on the sector $\Sigma_{\phi}$. \par
'$(2) \Rightarrow (1)$': For this direction, let $\pi / 2 < \phi < \vartheta < \theta$ and notice that for $z \in \Sigma_{\phi - \pi / 2}$ the semigroup has a representation via the Cauchy formula
\begin{align}
\label{Eq: Cauchy formula}
 S(z) = \frac{1}{2 \pi \mathrm{i}} \int_{\gamma_{\lvert z \rvert}} \mathrm{e}^{z \lambda} (\lambda + \CA)^{-1} \; \mathrm{d} \lambda,
\end{align}
where $\gamma_{\lvert z \rvert} = - \gamma_1 - \gamma_2 + \gamma_3$ is the path given by
\begin{align*}
 \gamma_{1 / 3} : [\lvert z \rvert^{-1} , \infty) \to \BC, \quad \gamma_{1 / 3} (t) := t \mathrm{e}^{\pm \mathrm{i} \vartheta} \qquad \text{and} \qquad \gamma_2 : [- \vartheta , \vartheta] \to \BC, \quad \gamma_2 (t) := \lvert z \rvert^{-1} \mathrm{e}^{\mathrm{i} t}.
\end{align*}
For $x \in \CX$ the estimate~\eqref{Eq: Abstract semigroup estimate} follows now by estimating~\eqref{Eq: Cauchy formula} by virtue of~\eqref{Eq: Abstract resolvent estimate}.
\end{proof}

\begin{corr}
\label{Cor: Lp-Lq gradient estimates whole space}
Let $\theta \in (0 , \pi)$. For all $1 < p \leq q < \infty$ with $\sigma := n (1 / p - 1 / q) / 2 < 1 / 2$ there exists a constant $C > 0$ such that for all $\lambda \in \Sigma_{\theta}$ and $f \in \L^p (\BR^n ; \BC^n)$ it holds
\begin{align*}
  \| \nabla (\lambda + A_{p , \BR^n})^{-1} \BP_{p , \BR^n} f \|_{\L^q (\BR^n ; \BC^{n^2})} \leq C \lvert \lambda \rvert^{\sigma - 1 / 2} \| f \|_{\L^p (\BR^n ; \BC^n)}.
\end{align*}
\end{corr}

\begin{proof}
This follows by combining~\cite[Eq.~$(2.3^{\prime})$]{K} with Proposition~\ref{Prop: Transference}. Notice that the semigroup estimate in~\cite[Eq.~$(2.3^{\prime})$]{K} is only proved for real values of $t$ but that it also holds for complex values $z \in \Sigma_{\phi}$ for each $0 < \phi < \pi / 2$.
\end{proof}

\begin{corr}
\label{Cor: Lp-Lq gradient estimates Lipschitz domain}
Let $D \subset \BR^n$ be a bounded Lipschitz domain and $\theta \in (0 , \pi)$. Then there exists $\varepsilon > 0$ such that for all $1 < p \leq q < \infty$ that satisfy~\eqref{Eq: Condition on p on bounded domains} and $\sigma := n (1 / p - 1 / q) / 2 < 1 / 2$  there exists a constant $C > 0$ such that for all $\lambda \in \Sigma_{\theta}$ and $f \in \L^p (D ; \BC^n)$ it holds
\begin{align*}
  \| \nabla (\lambda + A_{p , D})^{-1} \BP_{p , D} f \|_{\L^q (D ; \BC^{n^2})} \leq C \lvert \lambda \rvert^{\sigma - 1 / 2} \| f \|_{\L^p (D ; \BC^n)}.
\end{align*}
\end{corr}

\begin{proof}
This follows by combining~\cite[Cor.~1.2]{Tol} with Proposition~\ref{Prop: Transference}. Notice that the semigroup estimate in~\cite[Cor.~1.2]{Tol} is only proved for real values of $t$ but that it also holds for complex values $z \in \Sigma_{\phi}$ for each $0 < \phi < \pi / 2$.
\end{proof}

\section{A pressure estimate on bounded Lipschitz domains}
\label{sec-3}
\noindent
To get access to the $\L^p$-norm of the pressure, we introduce the Bogovski\u{\i} operator, which was constructed by Bogovski\u{\i}~\cite{Bog}, see also Galdi~\cite[Sec.~III.3]{G}. For this purpose, let
\begin{align*}
\L^p_0(D) := \Set{ F\in \L^p(D) | \int_{D} F\dx=0 }.
\end{align*}

\begin{prop}\label{prop-3.4}
Let $D$ be a bounded Lipschitz domain in $\BR^n$, $1 < p < \infty$, and $k \in \BN$.
Then there exists a continuous operator
\begin{align*}
\CB: \L^p (D) \to \W^{1,p}_0 (D ; \BC^n) \qquad \text{with} \qquad \CB \in \CL (\W^{k,p}_0 (D) , \W^{k + 1,p}_0 (D ; \BC^n))
\end{align*}
such that
\begin{align}
\dv (\CB g) = g \qquad (g \in \L^p_0 (D)).
\end{align}
\end{prop}

For purposes that come up in the following section, we record the following lemma to treat the operator $\CB$ in Sobolev spaces of negative order.
This was proven by Geissert, Heck, and Hieber~\cite[Thm.~2.5]{GHHa}.

\begin{prop}
Let $D$ be a bounded Lipschitz domain in $\BR^n$ and $1 < p < \infty$. Then the operator $\CB$ defined in Proposition~\ref{prop-3.4} extends to a bounded operator from $\W^{-1,p}_0 (D)$ to $\L^p (D ; \BC^n)$.
Here, the space $\W^{-1,p}_0 (D)$ denotes the dual space of $\W^{1,p'} (D)$.
\end{prop}

For $f \in \L^p_\sigma(D)$ consider the equation
\begin{align}
\label{eq-NS-L}
\left\{\begin{aligned}
\lambda u - \Delta u + \nabla\pi & = f && \text{in $D$,} \\
\dv (u) & = 0 && \text{in $D$,} \\
u & = 0 && \text{on $\pd D$}.
\end{aligned}\right.
\end{align}
We next turn to proving a decay estimate in $\lambda$ for the pressure term $\pi$.

\begin{prop}
\label{prop-3.6}
Let $D$ be a bounded Lipschitz domain in $\BR^n$ and $\theta \in (0 , \pi)$. Let $(u_\lambda, \pi_\lambda)$ be the unique solution
to the problem~\eqref{eq-NS-L} such that
$u_\lambda \in \CD (A_{p, D})$ and $\pi_\lambda \in \L^p_0 (D)$.
Define the operator
\begin{align*}
P_\lambda : \L^p_\sigma (D) \to \L^p_0 (D), \quad P_\lambda f:=\pi_\lambda.
\end{align*}
Then there exist positive constants $\varepsilon, C > 0$ and $\delta \in (0 , 1)$, such that for all numbers $p$ satisfying the condition~\eqref{Eq: Condition on p on bounded domains} and all numbers $\alpha$ that satisfy
\begin{align}
\label{alpha}
 0 \leq 2 \alpha < 1 - \frac{1}{p} \quad \text{if} \quad p \geq \frac{2}{1 + \delta} \qquad \text{and} \qquad 0 \leq 2 \alpha < 2 - \frac{3}{p} + \delta \quad \text{if} \quad p < \frac{2}{1 + \delta}
\end{align}
the estimate
\begin{align*}
\CR_{\L^p_\sigma (D) \to \L^p_0 (D)} \big\{ \lvert \lambda \rvert^\alpha P_\lambda \mid \lambda \in \Sigma_\theta \big\} \leq C
\end{align*}
holds.
\end{prop}

The proof of Proposition~\ref{prop-3.6} relies on mapping properties of the Helmholtz projection on bounded Lipschitz domains. These mapping properties, which are stated in Lemma~\ref{Lem: Mapping properties of Helmholtz and fractional power domains} are a reformulation of~\cite[Prop.~2.16]{MM}. To arrive at this reformulation recall that, by virtue of~\cite[Thm.~6.6.9]{Haase}, the boundedness of the $\H^{\infty}$-calculus, see~\cite[Thm.~16]{KW}, ensures the complex interpolation identity
\begin{align*}
 \CD(A_{p , D}^{s / 2}) = \big[ \CD(A_{p , D}^0) , \CD(A_{p , D}^{1 / 2}) \big]_{s} \qquad (s \in (0 , 1)).
\end{align*}
Finally, the facts $\CD(A_{p , D}^0) = \L^p_{\sigma} (D)$ and $\CD(A_{p , D}^{1 / 2}) = \W^{1 , p}_{0 , \sigma} (D)$, see Proposition~\ref{prop-3.2}, together with the interpolation result~\cite[Thm.~2.12]{MM} ensure that for $0 \leq s < 1 / p$ it holds
\begin{align*}
 \big[ \CD(A_{p , D}^0) , \CD(A_{p , D}^{1 / 2}) \big]_{s} = \big[ \L^p_{\sigma} (\Omega) , \W^{1 , p}_{0 , \sigma} (\Omega) \big]_{s} = \H^{s , p}_{\sigma} (\Omega).
\end{align*}
Altogether, this argument gives the following lemma.

\begin{lemm}
\label{Lem: Mapping properties of Helmholtz and fractional power domains}
Let $D \subset \BR^n$ be a bounded Lipschitz domain. Then there exists $\delta \in (0 , 1)$ and $\varepsilon > 0$ such that for all numbers $p$ that satisfy~\eqref{Eq: Condition on p on bounded domains} and all $s$ subject to
\begin{align*}
 0 \leq s < \frac{1}{p} \quad \text{if} \quad p \leq \frac{2}{1 - \delta} \qquad \text{and} \qquad 0 \leq s < \frac{3}{p} - 1 + \delta \quad \text{if} \quad \frac{2}{1 - \delta} < p,
\end{align*}
the Helmholtz projection $\BP_{p , D}$ restricts to a bounded operator 
\begin{align*}
 \BP_{p , D} : \H^{s , p} (D ; \BC^n) \to \CD(A_{p , D}^{s / 2}).
\end{align*}
\end{lemm}

Now, we are in the position to present a proof of Proposition~\ref{prop-3.6}.

\begin{proof}[Proof of Proposition~\ref{prop-3.6}]
Let $N \in \BN$, $\lambda_j \in \Sigma_\theta$, and $f_j \in \L^p_\sigma (D)$ $(j = 1 , \dots , N)$. Let $(u_j, \pi_j)$ be the solutions to the equation
\begin{align*}
\left\{\begin{aligned}
\lambda_j u_j - \Delta u_j + \nabla \pi_j & = f_j && \text{in $D$,}\\
\dv (u_j) & = 0 && \text{in $D$,}\\
u_j & = 0 && \text{on $\pd D$}
\end{aligned}\right.
\end{align*}
with $u_j \in \CD (A_{p, D})$ and $\pi_j \in \L^p_0 (D)$ being the pressure associated to $u_j$.
By virtue of Proposition~\ref{prop-3.4}, followed by the identity $A_{p, D} u_j = - \Delta u_j + \nabla \pi_j$
in the sense of distributions it follows for $0 < t < 1$
\begin{equation}
\label{3.7}
\begin{split}
\Big\| \sum_{j = 1}^{N} r_j (t) \lvert \lambda_j \rvert^\alpha \pi_j \Big\|_{\L^p (D)} & = \sup_{\substack{g \in \L^{p'}_0 (D)\\ \| g \|_{\L^{p'} (D)} \leq 1}}
\bigg| \int_D \bigg| \int_D \sum_{j = 1}^{N} r_j (t) \lvert \lambda_j \rvert^\alpha \pi_j \overline{\dv (\CB g)} \dx \bigg| \\
& \leq \sup_{\substack{g \in \L^{p'}_0 (D) \\ \|g\|_{\L^{p'}(D)}\leq1}}
\bigg| \int_D \bigg\langle \sum_{j = 1}^{N} r_j (t) \lvert \lambda_j \rvert^\alpha  A_{p, D} u_j, \CB g \bigg\rangle \dx \bigg|
\\
 &\qquad + \sup_{\substack{g \in \L^{p'}_0 (D) \\ \| g \|_{\L^{p'} (D)} \leq 1}}
\bigg| \int_D \bigg\langle \sum_{j = 1}^{N} r_j (t) \lvert \lambda_j \rvert^\alpha \nabla u_j, \nabla\CB g \bigg\rangle \dx \bigg|.
\end{split}
\end{equation}
Since the Helmholtz projection $\BP_{p , D}$ is the identity on $\L^p_\sigma(D)$, we obtain by duality by virtue of Proposition~\ref{prop-3.4} and Lemma~\ref{Lem: Mapping properties of Helmholtz and fractional power domains} that
\begin{equation}
\label{3.8}
\begin{split}
& \sup_{\substack{g \in \L^{p'}_0 (D) \\ \| g \|_{\L^{p'} (D)} \leq 1}}
\bigg| \int_D \bigg\langle \sum_{j = 1}^{N} r_j (t) \lvert \lambda_j \rvert^\alpha A_{p, D} u_j, \CB g \bigg\rangle \dx \bigg| \\
& = \sup_{\substack{g \in \L^{p'}_0 (D) \\ \| g \|_{\L^{p'} (D)} \leq 1}}
\bigg| \int_{D} \bigg\langle \sum_{j = 1}^{N} r_j (t) \lvert \lambda_j \rvert^\alpha A^{1 - \alpha}_{p, D} u_j, A^{\alpha}_{p' , D} \BP_{p' , D} \CB g \bigg\rangle \dx \bigg| \\
&\leq \bigg\|\sum_{j = 1}^{N}r_j (t) \lvert \lambda_j \rvert^\alpha A^{1 - \alpha}_{p, D} u_j \bigg\|_{\L^p_\sigma (D)}
\sup_{\substack{g \in \L^{p'}_0(D) \\ \| g \|_{\L^{p'} (D)} \leq 1}}
\|A^\alpha_{p' , D} \BP_{p' , D} \CB g \|_{\L^{p'}_{\sigma} (D)} \\
& \leq C \bigg\| \sum_{j = 1}^{N} r_j (t) \lvert \lambda_j \rvert^\alpha A^{1 - \alpha}_{p, D} u_j \bigg\|_{\L^p_\sigma (D)}
\end{split}
\end{equation}
for some constant $C > 0$ and any $\alpha$ satisfying the condition~\eqref{alpha}. By the estimate~\eqref{3.4}, we obtain
\begin{align*}
 \Big\| \sum_{j = 1}^{N} r_j (\cdot) \lvert \lambda_j \rvert^{\alpha} A^{1 - \alpha}_{p, D} u_j \Big\|_{\L^2(0 , 1 ; \L^p_{\sigma} (D))} \leq C \Big\| \sum_{j = 1}^{N} r_j (\cdot) f_j \Big\|_{\L^2 (0 , 1 ; \L^p_{\sigma} (D))}
\end{align*}
which, combined with~\eqref{3.8}, yields that
\begin{align*}
 \bigg\| \sup_{\substack{g \in \L^{p'}_0 (D) \\ \| g \|_{\L^{p'} (D)} \leq 1}}
\bigg| \int_D \bigg\langle \sum_{j = 1}^{N} r_j (\cdot) \lvert \lambda_j \rvert^{\alpha} A_{p, D} u_j, \CB g \bigg\rangle \; \dx \bigg| \; \bigg\|_{\L^2 (0 , 1)}
\leq C \Big\| \sum_{j = 1}^{N} r_j (\cdot) f_j \Big\|_{\L^2(0 , 1 ; \L^p_{\sigma} (D))}.
\end{align*}
In addition, from~\eqref{3.5} together with Proposition~\ref{prop-3.4} we have that
\begin{align*}
 \bigg\| \sup_{\substack{g \in \L^{p'}_0 (D) \\ \| g \|_{\L^{p'} (D)} \leq 1}}
\bigg|\int_D \bigg\langle \sum_{j = 1}^{N} r_j (\cdot) \lvert \lambda_j \rvert^{\alpha} \nabla u_j, \nabla \CB g \bigg\rangle \; \dx \bigg| \bigg\|_{\L^2(0 , 1)}
\leq C \Big\| \sum_{j = 1}^{N} r_j (\cdot) f_j \Big\|_{\L^2(0 , 1 ; \L^p_{\sigma} (D))}
\end{align*}
for $\alpha$ satisfying~\eqref{alpha} and $\lambda_j \in \Sigma_\theta$. In view of~\eqref{3.7} this completes the proof.
\end{proof}

\section{The Stokes operator in exterior Lipschitz domains}
\label{sec-4}
\noindent
This section is devoted to the proofs of Theorems~\ref{th-1.1},~\ref{Thm: Maximal regularity}, and~\ref{Thm: Navier-Stokes}. The proof of these facts relies on the philosophy that the solution to the Stokes resolvent problem can ``almost'' be written as the sum of a solution to a whole space problem and a solution to a problem on an appropriately chosen bounded Lipschitz domain. In view of this, we follow the argument of Geissert \textit{et al}.~\cite{GHHSK} for large resolvent parameters $\lambda$ and perform a refined analysis that resembles in some parts to the argument of Iwashita~\cite{Iwashita} for small values of $\lambda$. \par
Choose $R > 0$ sufficiently large such that $\Omega^c \subset B_R (0) = \{x \in \BR^n \mid |x| < R \}$ and
define
\begin{equation*}
\begin{split}
D & := \Omega \cap B_{R + 5} (0),\\
K_1 & := \{x \in \Omega \mid R < |x| < R + 3 \},\\
K_2 & := \{x \in \Omega \mid R + 2 < |x| < R + 5 \}.
\end{split}
\end{equation*}
Let $\CB_1$ and $\CB_2$ be the Bogovski\u{\i} operators, introduced in Proposition~\ref{prop-3.4}, defined in the domain $K_1$ and $K_2$, respectively.
In addition, let $\varphi, \eta \in \C^\infty (\BR^n)$ be cut-off functions such that $0 \leq \varphi,\eta \leq 1$ and
\begin{align*}
\varphi (x) & =
\begin{cases}
0 & \text{for $|x| \leq R + 1$}, \\
1 & \text{for $|x| \geq R + 2$},
\end{cases} \\
\eta (x) & =
\begin{cases}
1 & \text{for $|x| \leq R + 3$}, \\
0 & \text{for $|x| \geq R + 4$}.
\end{cases}
\end{align*}
For $f \in \L^p (\Omega ; \BC^n)$ denote by $f^R$ the zero extension of $f$ to $\BR^n$.
Notice that $f \in \L^p_{\sigma} (\Omega)$ implies $f^R \in \L^p_\sigma (\BR^n)$ because $\C^\infty_{c, \sigma} (\Omega)$ is dense in $\L^p_\sigma (\Omega)$.
Set $f^D = \eta f - \CB_2 ((\nabla \eta) \cdot f )$, where $\CB_2 ((\nabla \eta) \cdot f )$ is regarded as a function that is extended by zero to all of $\BR^n$. Notice that $f \in \L^p_{\sigma} (\Omega)$ implies $\int_{K_2} (\nabla \eta) \cdot f \dx = 0$ and thus in this case that $f^D \in \L^p_\sigma (D)$. \par

In the following, we will agree upon the following convention for $\varepsilon$ and $p$.

\begin{convention}
\label{Conv}
Let $\varepsilon > 0$ and $p$ be such that the conditions in Proposition~\ref{Prop: Helmholtz on exterior} and~\eqref{Eq: Condition on p on bounded domains} are satisfied. 
\end{convention}

Let $\theta \in (0 , \pi)$ and $\varepsilon > 0$ and $p$ be subject to Convention~\ref{Conv}. For $\lambda \in \Sigma_\theta$ there exist by Propositions~\ref{prop-2.5} and~\ref{prop-3.2} functions $u^R_\lambda$, $u^D_\lambda$, and $\pi^D_\lambda$ satisfying the equations
\begin{align}
\label{eq-4.1}
\left\{\begin{aligned}
\lambda u^R_\lambda - \Delta u^R_\lambda + \nabla g & = f^R && \text{in } \BR^n, \\
\dv (u^R_\lambda) & = 0 && \text{in }\BR^n,
\end{aligned}\right.
\end{align}
and
\begin{align}
\label{eq-4.2}
\left\{\begin{aligned}
\lambda u^D_\lambda - \Delta u^D_\lambda + \nabla\pi^D_\lambda & = f^D && \text{in } D, \\
\dv (u^D_\lambda) & = 0 && \text{in } D, \\
u^D_\lambda & = 0 && \text{on } \pd D.
\end{aligned}\right.
\end{align}
Recall, that $g$ was given by $\nabla g = (\I - \BP_{p , \BR^n}) f^R$. In the following, we normalize $g$ to satisfy
\begin{align}
\label{Eq: Normalization of g}
 \int_D g \dx = 0.
\end{align}
The operator $U_{\lambda} : \L^p (\Omega ; \BC^d) \to \L^p_\sigma (\Omega)$ defined next is ``almost'' the solution operator to the resolvent problem on the exterior domain $\Omega$ with right-hand side $f$. Define $U_{\lambda}$ by
\begin{align}
\label{Eq: The approximate solution}
U_\lambda f := \varphi u^R_\lambda + (1 - \varphi) u^D_\lambda - \CB_1 \big( (\nabla \varphi) \cdot (u^R_\lambda - u^D_\lambda) \big),
\end{align}
where $u^R_\lambda$ and $u^D_\lambda$ are the functions satisfying equations~\eqref{eq-4.1} and~\eqref{eq-4.2}, respectively. Again $\CB_1 ( (\nabla \varphi) \cdot (u^R_\lambda - u^D_\lambda) )$ is regarded as the extension by zero to the whole space. Notice that even though the regularity theory of solutions to the Stokes equations on bounded Lipschitz domains does not allow for $\W^{2 , p}$-regularity of $u_{\lambda}^D$ on $D$, standard inner regularity results, see Galdi~\cite[Thm.~IV.4.1]{G}, yield that $(\nabla \varphi) \cdot (u_{\lambda}^R - u_{\lambda}^D) \in \W^{2 , p}_0 (K_1)$ so that by Proposition~\ref{prop-3.4} the extension by zero of $\CB_1 ( (\nabla \varphi) \cdot (u^R_\lambda - u^D_\lambda) )$ lies in $\W^{3 , p} (\BR^n)$. \par
Moreover, Propositions~\ref{prop-2.5} and~\ref{prop-3.2} imply that $\varphi u^R_\lambda \in \W^{2, p} (\Omega) \cap \W^{1, p}_0 (\Omega)$ and $(1 - \varphi) u^D_\lambda \in \W^{1, p}_0 (\Omega)$.
Thus, we observe that for all $p$ subject to Convention~\ref{Conv} it holds
\begin{align*}
U_\lambda f \in \W^{1,p}_0(\Omega)\cap \L^p_\sigma(\Omega).
\end{align*}

Setting $\Pi_\lambda f := (1 - \varphi) \pi^D_\lambda + \varphi g$, the pair $(U_\lambda f, \Pi_\lambda f)$ satisfies the equation
\begin{align}
\label{eq-4.5}
\left\{\begin{aligned}
(\lambda - \Delta) U_\lambda f + \nabla \Pi_\lambda f & = f + T_\lambda f && \text{in } \Omega, \\
\dv (U_\lambda f) & = 0 && \text{in } \Omega, \\
U_\lambda f & = 0 && \text{on } \pd \Omega
\end{aligned}\right.
\end{align}
in the sense of distributions. Here, $T_\lambda$ is given by
\begin{equation}
\label{4.6}
\begin{split}
T_\lambda f
& := -2[(\nabla \varphi) \cdot \nabla] (u^R_\lambda - u^D_\lambda) - (\Delta \varphi)(u^R_\lambda - u^D_\lambda) \\
& \quad + (\nabla \varphi)(g - \pi^D_\lambda) - (\lambda - \Delta) \CB_1 \big( (\nabla \varphi) \cdot (u^R_\lambda - u^D_\lambda) \big).
\end{split}
\end{equation}
Observe the following two properties concerning the operator $T_{\lambda}$ defined on $\L^p (\Omega ; \BC^n)$ where $p$ satisfies~\eqref{Eq: Condition on p on bounded domains}. First of all, for each function $f \in \L^p (\Omega ; \BC^n)$, the support of $T_{\lambda} f$ lies in the \textit{compact} set $\overline{K_1}$. Second, notice that inner regularity results for the Stokes equations, see~\cite[Thm.~IV.4.1]{G}, imply that $T_{\lambda}$ is a bounded operator from $\L^p (\Omega ; \BC^n)$ to $\W^{1 , p} (\Omega ; \BC^n)$. Thus, in combination with the support property, $T_{\lambda}$ turns out to be a compact operator. This is recorded in the following lemma.

\begin{lemm}
\label{Lem: Properties of T_lambda}
Let $\theta \in (0 , \pi)$, $\lambda \in \Sigma_{\theta}$, and let $p$ be subject to Convention~\ref{Conv}. Then $T_{\lambda}$ satisfies $T_{\lambda} \in \CL(\L^p (\Omega ; \BC^n) , \W^{1 , p} (\Omega ; \BC^n))$, it satisfies for each $f \in \L^p (\Omega ; \BC^n)$ the property $\supp(T_{\lambda} f) \subset \overline{K_1}$, and it is compact on $\L^p (\Omega ; \BC^n)$.
\end{lemm}

The further line of action will be split into five consecutive steps. The first step is dedicated to the investigation of the operator $f \mapsto U_{\lambda} f$. Here, estimates with respect to $\lambda$ are established. To obtain estimates to the Stokes resolvent problem on the exterior domain $\Omega$ by means of the operator $U_{\lambda}$ the operator $\I + T_{\lambda}$ and its solenoidal counterpart $\I + \BP_{p , \Omega} T_{\lambda}$ have to be analyzed. In the second step we show that $\BP_{p , \Omega} T_{\lambda}$ regarded as an operator on $\L^p_{\sigma} (\Omega)$ is small for large values of $\lambda$ so that $\I + \BP_{p , \Omega} T_{\lambda}$ can be inverted by a simple Neumann series argument. The third step is much more delicate as here continuity properties of the operator $T_{\lambda}$ for small values of $\lambda$ are studied. In particular, we will show that $T_{\lambda}$ has a well-defined limit as $\lambda \to 0$ that is a compact operator. In the fourth step, the invertibility of $\I + T_{\lambda}$ for small values of $\lambda$ is proven by standard Fredholm theory and a perturbation argument. In the final fifth step, everything will be combined to give the proofs of Theorems~\ref{th-1.1},~\ref{Thm: Maximal regularity}, and~\ref{Thm: Navier-Stokes}.

\subsection*{Step~1: Investigation of the operator $U_{\lambda}$}

We start by giving bounds on the operator norms of the operators $U_{\lambda}$ and $\nabla U_{\lambda}$ in terms of the resolvent parameter $\lambda$.

\begin{lemm}
\label{Lem: Perturbed solution operator}
Let $\Omega \subset \BR^n$ be an exterior Lipschitz domain and $\theta \in (0 , \pi)$. Let $\varepsilon > 0$ and $p \leq q$ satisfy Convention~\ref{Conv} and $\sigma := n ( 1 / p - 1 / q ) / 2 \leq 1 / 2$. Then there exists a constant $C > 0$ such that
\begin{align}
\label{Eq: Ineq 1}
 \CR_{\L^p (\Omega ; \BC^n) \to \L^q_{\sigma} (\Omega)}  \big\{ \lvert \lambda \rvert^{1 - \sigma} U_{\lambda} \mid\ \lambda \in \Sigma_{\theta} \big\} \leq C
\end{align}
and
\begin{align}
\label{Eq: Ineq 2}
 \| U_1 f \|_{\W^{1 , p}_{0 , \sigma} (\Omega)} \leq C \| f \|_{\L^p (\Omega ; \BC^n)}.
\end{align}
If additionally $q < n$, $p < n / 2$, and $\sigma < 1 / 2$, then there exists $C > 0$ such that
\begin{align*}
 \lvert \lambda \rvert^{1 / 2 - \sigma} \| \nabla U_{\lambda} f \|_{\L^q (\Omega ; \BC^{n^2})} \leq C \| f \|_{\L^p (\Omega ; \BC^n)} \qquad (\lambda \in \Sigma_{\theta} , f \in \L^p (\Omega ; \BC^n)).
\end{align*}
\end{lemm}

\begin{proof}
First of all, recall the definition of $U_{\lambda}$ in~\eqref{Eq: The approximate solution}. To prove~\eqref{Eq: Ineq 1} let $N \in \BN$, $\lambda_j \in \Sigma_{\theta}$, and $f_j \in \L^p (\Omega ; \BC^n)$ where $1 \leq j \leq N$. An application of Lemmas~\ref{Lem: R-bounded L^p-L^q estimates whole space} and~\ref{Lem: R-bounded L^p-L^q estimates Lipschitz domain} together with the boundedness of Bogovski\u{\i}'s operator from $\L^q_0 (K_1)$ to $\W^{1 , q}_0 (K_1 ; \BC^n)$, see Proposition~\ref{prop-3.4}, imply the estimate
\begin{align*}
 &\Big\| \sum_{j = 1}^N \lvert \lambda_j \rvert^{1 - \sigma} U_{\lambda_j} r_j (\cdot) f_j \Big\|_{\L^2(0 , 1 ; \L^q_{\sigma} (\Omega))} \\
 &\qquad \leq C \bigg( \Big\| \sum_{j = 1}^N \lvert \lambda_j \rvert^{1 - \sigma} r_j (\cdot) u_{\lambda_j}^R \Big\|_{\L^2(0 , 1 ; \L^q_{\sigma} (\BR^n))} + \Big\| \sum_{j = 1}^N \lvert \lambda_j \rvert^{1 - \sigma} r_j (\cdot) u_{\lambda_j}^D \Big\|_{\L^2(0 , 1 ; \L^q_{\sigma} (D))} \bigg) \\
 &\qquad \leq C \Big\| \sum_{j = 1}^N r_j (\cdot) f_j \Big\|_{\L^2(0 , 1 ; \L^p (\Omega ; \BC^n))}.
\end{align*}
Concerning~\eqref{Eq: Ineq 2}, this follows by the product rule and by the boundedness estimates given in Proposition~\ref{prop-2.5}, Proposition~\ref{prop-3.2}, Lemma~\ref{lem-3.7}, and Proposition~\ref{prop-3.4}. \par
Now, let $\lambda \in \Sigma_{\theta}$ and $f \in \L^p (\Omega ; \BC^n)$ and let $p$ and $q$ additionally satisfy $q < n$ and $p < n / 2$ with $\sigma < 1 / 2$. To derive an estimate to $\nabla U_{\lambda}$, notice that the support of $\nabla \varphi$ is contained in the annulus $\mathfrak{A} := \overline{B_{R + 2} (0)} \setminus B_{R + 1} (0)$. Thus, applying H\"older's inequality followed by Sobolev's inequality together with the boundedness of Bogovski\u{\i}'s operator from $\L^q_0 (K_1)$ to $\W^{1 , q}_0 (K_1 ; \BC^n)$ imply with $q^* := nq / (n - q)$
\begin{align*}
 \| \nabla U_{\lambda} f \|_{\L^q(\Omega ; \BC^{n^2})} &\leq \| \nabla u_{\lambda}^R \|_{\L^q(\BR^n ; \BC^{n^2})} + \| \nabla u_{\lambda}^D \|_{\L^q(D ; \BC^{n^2})} + \| u_{\lambda}^R \|_{\L^q (\mathfrak{A} ; \BC^n)} + \| u_{\lambda}^D \|_{\L^q (\mathfrak{A} ; \BC^n)} \\
 &\leq \| \nabla u_{\lambda}^R \|_{\L^q(\BR^n ; \BC^{n^2})} + \| \nabla u_{\lambda}^D \|_{\L^q(D ; \BC^{n^2})} + C \Big( \| u_{\lambda}^R \|_{\L^{q^*} (\BR^n ; \BC^n)} + \| u_{\lambda}^D \|_{\L^{q^*} (D ; \BC^n)} \Big) \\
 &\leq C \Big( \| \nabla u_{\lambda}^R \|_{\L^q(\BR^n ; \BC^{n^2})} + \| \nabla u_{\lambda}^D \|_{\L^q(D ; \BC^{n^2})} \Big).
\end{align*}
Finally, Corollaries~\ref{Cor: Lp-Lq gradient estimates whole space} and~\ref{Cor: Lp-Lq gradient estimates Lipschitz domain} imply
\begin{align*}
 \lvert \lambda \rvert^{1 / 2 - \sigma} \| \nabla U_{\lambda} f \|_{\L^q(\Omega ; \BC^{n^2})} \leq C \| f \|_{\L^p (\Omega ; \BC^n)}. & \qedhere
\end{align*}
\end{proof}


\subsection*{Step~2: Invertibility of $\I + \BP_{p , \Omega} T_{\lambda}$ for large values of $\lambda$}

To obtain decay of the family of operators $T_{\lambda}$ with respect to $\lambda$ it is essential to project these operators onto $\L^p_{\sigma} (\Omega)$, {\em i.e.}, to consider
\begin{align*}
 \BP_{p , \Omega} T_{\lambda} : \L^p_{\sigma} (\Omega) \to \L^p_{\sigma} (\Omega).
\end{align*}
This has the effect that non-decaying terms with respect to $\lambda$, that is $g$ (see~\eqref{4.6}) and a certain term within $\pi_{\lambda}^D$ that exists for non-solenoidal right-hand sides, are eliminated. 

\begin{lemm}
\label{lem-4.1}
Let $\Omega$ be an exterior Lipschitz domain in $\BR^n$ and $\theta \in (0 , \pi)$. Then, for $\varepsilon$ and $p$ subject to Convention~\ref{Conv} and $\alpha$ being a constant satisfying~\eqref{alpha} there exists $C > 0$ satisfying for all $\lambda^* \geq 1$
\begin{align*}
\CR_{\L^p_\sigma (\Omega)} \big\{ \lvert \lambda \rvert^\alpha \BP_{p , \Omega} T_\lambda \mid \lambda \in \Sigma_\theta , \lvert \lambda \rvert \geq \lambda^* \big\} \leq C.
\end{align*}
\end{lemm}

\begin{proof}
Let $f \in \L^p_\sigma (\Omega)$ and $u^R_\lambda$ and $u^D_\lambda$ be the functions satisfying the equations~\eqref{eq-4.1} and~\eqref{eq-4.2}, respectively.
Since $f^R \in \L^p_{\sigma} (\BR^n)$ the function $g$ in~\eqref{eq-4.1} is zero by Proposition~\ref{prop-2.5}. Thus, we have
\begin{align}
\label{Eq: Relation for Bogovskii}
\lambda u^R_\lambda - \lambda u^D_\lambda = f^R + \Delta u^R_\lambda - (f^D + \Delta u^D_\lambda -\nabla \pi^D_\lambda).
\end{align}
Since $\supp(\nabla \varphi) \cap K_2 = \emptyset$ and $\eta \equiv 1$ on $\supp(\nabla \varphi)$, the definitions of $f^R$ and $f^D$ further yield
\begin{align*}
 (\nabla \varphi) \cdot (f^R - f^D) = (\nabla \varphi) (1 - \eta) f = 0.
\end{align*}
This combined with~\eqref{Eq: Relation for Bogovskii} results in
\begin{align*}
\lambda \CB_1 \big( (\nabla \varphi)  \cdot(u^R_\lambda - u^D_\lambda) \big)
= \CB_1 \big( (\nabla \varphi) \cdot (\Delta u^R_\lambda - \Delta u^D_\lambda) \big)
+ \CB_1 \big( (\nabla \varphi) \cdot (\nabla \pi^D_\lambda) \big).
\end{align*}
From this fact, we rewrite $\BP_{p , \Omega} T_\lambda$ as
\begin{equation*}
\begin{split}
\BP_{p , \Omega} T_\lambda f
& = - 2 \big[ \BP_{p , \Omega} [(\nabla \varphi) \cdot \nabla] (u^R_\lambda - u^D_\lambda)\big] - \big[ \BP_{p , \Omega} (\Delta \varphi) (u^R_\lambda - u^D_\lambda) \big]
+ \big[ \BP_{p , \Omega} \Delta \CB_1 \big((\nabla \varphi) \cdot (u^R_\lambda - u^D_\lambda) \big) \big] \\
& \quad - \big[ \BP_{p , \Omega} \CB_1\big((\nabla \varphi) \cdot (\Delta u^R_\lambda - \Delta u^D_\lambda) \big) \big]
- \big[ \BP_{p , \Omega} \big((\nabla \varphi)\pi^D_\lambda \big)\big] - \big[ \BP_{p , \Omega} \CB_1 \big( (\nabla \varphi) \cdot (\nabla \pi^D_\lambda) \big)  \big] \\
& \equiv T^1_\lambda f + T^2_\lambda f + T^3_\lambda f + T^4_\lambda f + T^5_\lambda f + T^6_\lambda f.
\end{split}
\end{equation*}
By virtue of Proposition~\ref{Prop: Helmholtz on exterior}, Proposition~\ref{prop-2.5}, and Lemma~\ref{lem-3.7} there exists $C > 0$ such that
\begin{align*}
 \CR_{\L^p_{\sigma} (\Omega)} \big\{ \lvert \lambda \rvert^{1 / 2} T_{\lambda}^1 \mid \lambda \in \Sigma_{\theta} , \lvert \lambda \rvert \geq \lambda^* \big\} \leq C.
\end{align*}
Now, use Kahane's contraction principle, see Proposition~\ref{Prop: Kahane}, and the fact that $\alpha < 1 / 2$ to deduce that
\begin{align*}
 \CR_{\L^p_{\sigma} (\Omega)} \big\{ \lvert \lambda \rvert^{\alpha} T_{\lambda}^1 \mid \lambda \in \Sigma_{\theta} , \lvert \lambda \rvert \geq \lambda^* \big\} \leq 2 (\lambda^*)^{\alpha - 1 / 2} C,
\end{align*}
where $C > 0$ is the constant from the previous estimate. Similarly, the operator families $T_{\lambda}^2$ and $T_{\lambda}^5$ are estimated, but relying additionally on Propositions~\ref{prop-3.2} and~\ref{prop-3.6}. To estimate, $T_{\lambda}^3$ use the boundedness of $\CB_1 : \W^{1 , p}_0 (K_1) \to \W^{2 , p}_0 (K_1 ; \BC^n)$ stated in Proposition~\ref{prop-3.4} and proceed as for $T_{\lambda}^1$ and $T_{\lambda}^2$. \par
Finally, we present the estimates for $T_{\lambda}^4$ and remark that $T_{\lambda}^6$ is estimated similarly. Let $N \in \BN$, $\lambda_j \in \Sigma_{\theta}$ with $\lvert \lambda_j \rvert \geq \lambda^*$, and $f_j \in \L^p_{\sigma} (\Omega)$ $(j = 1 , \dots , N)$. Then
\begin{align*}
 (\nabla \varphi) \cdot (\Delta u^R_{\lambda_j} - \Delta u^D_{\lambda_j}) = \dv \Big( \sum_{i = 1}^n \partial_i \varphi \nabla [ (u_{\lambda_j}^R)_i - (u_{\lambda_j}^D)_i ] \Big) - \nabla^2 \varphi : \nabla (u_{\lambda_j}^R - u_{\lambda_j}^D),
\end{align*}
where $A : B = \sum_{i , k = 1}^n A_{i k} B_{i k}$ for two $n \times n$ matrices $A$ and $B$. Consequently, the boundedness of $\BP_{p , \Omega}$ and the boundedness of $\CB_1 : \W^{-1 , p}_0 (K_1) \to \L^p_0 (K_1 ; \BC^n)$ and $\CB_1 : \L^p_0 (K_1) \to \W^{1 , p}_0 (K_1 ; \BC^n)$ yield
\begin{align*}
 \Big\| \sum_{j = 1}^N r_j (\cdot) \lvert \lambda_j \rvert^{1 / 2} T_{\lambda_j}^4 f_j \Big\|_{\L^2(0 , 1 ; \L^p_{\sigma} (\Omega))} \leq C \Big\| \sum_{j = 1}^N r_j (\cdot) f_j \Big\|_{\L^2(0 , 1 ; \L^p_{\sigma} (\Omega))}.
\end{align*}
Again, Kahane's contraction principle together with $\alpha < 1 / 2$ result in the estimate
\begin{align*}
 \CR_{\L^p_{\sigma} (\Omega)} \big\{ \lvert \lambda \rvert^{\alpha} T_{\lambda}^4 \mid \lambda \in \Sigma_{\theta} , \lvert \lambda \rvert \geq \lambda^* \big\} \leq 2 (\lambda^*)^{\alpha - 1 / 2} C. &\qedhere
\end{align*}
\end{proof}

\begin{corr}
\label{Cor: Control for large lambda}
Let $\Omega$ be an exterior domain in $\BR^n$ and $\theta \in (0 , \pi)$. Let $\varepsilon$ and $p$ be subject to Convention~\ref{Conv}. Then there exists $\lambda^* \geq 1$ such that for all $\lambda \in \Sigma_{\theta}$ with $\lvert \lambda \rvert \geq \lambda^*$ the operator
\begin{align*}
 \I + \BP_{p , \Omega} T_{\lambda} : \L^p_{\sigma} (\Omega) \to \L^p_{\sigma} (\Omega)
\end{align*}
is invertible. Moreover, $\lambda^*$ can be chosen such that it holds
\begin{align*}
 \CR_{\L^p_{\sigma} (\Omega)} \big\{ (\I + \BP_{p , \Omega} T_{\lambda})^{-1} \mid \lambda \in \Sigma_{\theta} , \lvert \lambda \rvert \geq \lambda^* \big\} \leq 2.
\end{align*}
\end{corr}

\begin{proof}
This follows by Lemma~\ref{lem-4.1} and a Neumann series argument combined with Proposition~\ref{Prop: Kahane}.
\end{proof}

\subsection*{Step~3: Continuity and continuation of $T_{\lambda}$ for small $\lambda$}

While it was (in the case of large values of $\lambda$) beneficial to consider the projected operator
\begin{align*}
 \I + \BP_{p , \Omega} T_{\lambda} : \L^p_{\sigma} (\Omega) \to \L^p_{\sigma} (\Omega)
\end{align*}
the question of invertibility for small values of $\lambda$ is resolved for the operator
\begin{align*}
 \I + T_{\lambda} : \L^p(\Omega ; \BC^n) \to \L^p(\Omega ; \BC^n).
\end{align*}
Here, decay properties of $T_{\lambda}$ are not the prevalent feature but essentially the fact that $T_{\lambda}$ is a regularizing and localizing operator. The locality will also be of importance in Step~4, where the injectivity of $\I + T_{\lambda}$ is shown. To obtain estimates for the Stokes resolvent up to $\lambda = 0$, we are going to define an operator $T_0$ as the limit of the operators $T_{\lambda}$ as $\lvert \lambda \rvert \searrow 0$ and unveil some of its properties. As a preparation we prove the following lemma. To this end, for $\lambda , \mu , \lambda_j \in \Sigma_{\theta}$ and $f , f_j \in \L^p (\BR^n ; \BC^n)$ we use the notation
\begin{align*}
 u_{\lambda}^R := (\lambda + A_{p , \BR^n})^{-1} \BP_{p , \BR^n} f, \quad  u_{\lambda_j}^R := (\lambda_j + A_{p , \BR^n})^{-1} \BP_{p , \BR^n} f_j, \quad \text{and} \quad u_{\mu , j}^R := (\mu + A_{p , \BR^n})^{-1} \BP_{p , \BR^n} f_j.
\end{align*}

\begin{lemm}
\label{Lem: Convergences for lambda to zero}
Let $\theta \in (0 , \pi)$, $1 < p < n / 2$, $f \in \L^p (\BR^n ; \BC^n)$, and $\nabla g = (\I - \BP_{p , \BR^n}) f$. There exists $u_0^R \in \W^{2 , p}_{\loc} (\BR^n ; \BC^n)$ with $\nabla^2 u_0^R \in \L^p (\BR^n ; \BC^{n^3})$ such that 
\begin{align*}
\left\{ \begin{aligned}
 - \Delta u_0^R + \nabla g &= f && \text{in } \BR^n \\
 \dv (u_0^R) &= 0 && \text{in } \BR^n,
\end{aligned} \right.
\end{align*}
and such that $u_{\lambda}^R \to u_0^R$ in $\W^{2 , p}_{\loc} (\BR^n ; \BC^n)$ and $\nabla^2 u_{\lambda}^R \to \nabla^2 u_0^R$ in $\L^p (\BR^n ; \BC^{n^3})$ as $\lambda \to 0$ with $\lambda \in \Sigma_{\theta}$. \par
Furthermore, there exist constants $0 < \alpha , \beta < 1$ such that for each open and bounded set $O \subset \BR^n$ there exists a constant $C > 0$ such that for all $N \in \BN$, $f_j \in \L^p (\BR^n ; \BC^n)$, $\lambda^* > 0$, and $\lambda_j , \mu \in B_{\lambda^*} (0) \cap \Sigma_{\theta}$ $(j = 1 , \dots , N)$ it holds
\begin{align}
 \Big\| \sum_{j = 1}^N r_j (\cdot) ( u_{\lambda_j}^R  - u_{\mu , j}^R) \Big\|_{\L^2 (0 , 1 ; \W^{1 , p} (O ; \BC^n))} &\leq C \max\{\lambda^* , 1\}^{\beta} \max_{1 \leq i \leq N} \min \big\{ \lvert \lambda_i \rvert^{\alpha - 1} \lvert \lambda_i - \mu \rvert , \lvert \mu \rvert^{\alpha - 1} \lvert \mu - \lambda_i \rvert \big\} \notag \\
 &\qquad \cdot \Big\| \sum_{j = 1}^N r_j (\cdot) f_j \Big\|_{\L^2 (0 , 1 ; \L^p (\BR^n ; \BC^n))}. \label{Eq: Quantified approximation whole space, general} 
\end{align}
In particular, if $u_{0 , j}^R$ denotes the limit obtained above but with datum $f_j$, then for each open and bounded set $O \subset \BR^n$ there exists a constant $C > 0$ such that for all $N \in \BN$, $f_j \in \L^p (\BR^n ; \BC^n)$, $\delta > 0$, and $\lambda_j \in B_{\delta} (0) \cap \Sigma_{\theta}$ $(j = 1 , \dots , N)$ it holds
\begin{align}
\label{Eq: Quantified approximation whole space, origin}
 \Big\| \sum_{j = 1}^N r_j (\cdot) ( u_{\lambda_j}^R  - u_{0 , j}^R) \Big\|_{\L^2 (0 , 1 ; \W^{1 , p} (O ; \BC^n))} \leq C \max\{ \delta^{\beta} , 1\} \delta^{\alpha} \Big\| \sum_{j = 1}^N r_j (\cdot) f_j \Big\|_{\L^2 (0 , 1 ; \L^p (\BR^n ; \BC^n))}.
\end{align}
\end{lemm}

\begin{proof}
Let $O \subset \BR^n$ be open and bounded. Let $\lambda^* > 0$, $N \in \BN$, $f_j \in \L^p (\BR^n ; \BC^n)$, and $\lambda_j , \mu \in \Sigma_{\theta}$ with $\lvert \mu \rvert , \lvert \lambda_j \rvert < \lambda^*$ $(j = 1 , \dots , N)$. Let $n p / (n - p) = p^* < q < \infty$ with $1 / p - 1 / q \leq 3 / n$. Denote by $q_* := q n / (q + n)$ and notice that $\W^{1 , q_*} (\BR^n) \subset \L^q (\BR^n)$. An application of H\"older's inequality followed by Sobolev's inequality together with the resolvent identity implies for almost every $t \in (0 , 1)$
\begin{align*}
 \Big\| \sum_{j = 1}^N r_j (t) \nabla \big( u_{\lambda_j}^R  - u_{\mu , j}^R \big) \Big\|_{\L^p (O ; \BC^{n^2})} \leq C \Big\| \sum_{j = 1}^N r_j (t) \nabla^2 (\mu + A_{p , \BR^n})^{-1} (\mu - \lambda_j) u_{\lambda_j}^R \Big\|_{\L^{q_*} (\BR^n ; \BC^{n^3})}.
\end{align*}
Now, Proposition~\ref{prop-2.5} ensures
\begin{align*}
 \Big\| \sum_{j = 1}^N r_j (\cdot) \nabla^2 (\mu + A_{p , \BR^n})^{-1} (\mu - \lambda_j) u_{\lambda_j}^R \Big\|_{\L^2 (0 , 1 ; \L^{q_*} (\BR^n ; \BC^{n^3}))} \leq C \Big\| \sum_{j = 1}^N r_j (\cdot) (\mu - \lambda_j) u_{\lambda_j}^R \Big\|_{\L^2 (0 , 1 ; \L^{q_*} (\BR^n ; \BC^n))}.
\end{align*}
Notice that $q_* > p$ and that $1 / p - 1 / {q_*} \leq 2 / n$. Thus, Lemma~\ref{Lem: R-bounded L^p-L^q estimates whole space} implies with $\sigma := n (1 / p - 1 / {q_*})/2$
\begin{align*}
 \Big\| \sum_{j = 1}^N r_j (\cdot) (\mu - \lambda_j) u_{\mu , j}^R \Big\|_{\L^2 (0 , 1 ; \L^{q_*} (\BR^n ; \BC^n))} \leq C \Big\| \sum_{j = 1}^N r_j (\cdot) \lvert \lambda_j \rvert^{\sigma - 1} (\mu - \lambda_j) f_j \Big\|_{\L^2 (0 , 1 ; \L^p (\BR^n ; \BC^n))}.
\end{align*}
Summarizing the previous estimates followed by Kahane's contraction principle delivers
\begin{align}
\label{Eq: Quantified approximation gradient}
 \Big\| \sum_{j = 1}^N r_j (\cdot) \nabla \big( u_{\lambda_j}^R  - u_{\mu , j}^R \big) \Big\|_{\L^2(0 , 1 ; \L^p (O ; \BC^{n^2}))} \leq C \sup_{1 \leq i \leq N} \lvert \lambda_i \rvert^{\sigma - 1} \lvert \lambda_i - \mu \rvert \Big\| \sum_{j = 1}^N r_j (\cdot) f_j \Big\|_{\L^2 (0 , 1 ; \L^p (\BR^n ; \BC^n))}.
\end{align}
Next, for any $n p / (n - 2 p) = p^{**} < q < \infty$ with $1 / p - 1 / q \leq 4 / n$ and $q_{**} := q n / (n + 2 q)$ define $\nu := n (1 / p - 1 / {q_{**}}) / 2$. Then analogously as above one finds
\begin{align}
\label{Eq: Quantified approximation}
 \Big\| \sum_{j = 1}^N r_j (\cdot) \big( u_{\lambda_j}^R  - u_{\mu , j}^R \big) \Big\|_{\L^2(0 , 1 ; \L^p (O ; \BC^{n}))} \leq C \sup_{1 \leq i \leq N} \lvert \lambda_i \rvert^{\nu - 1} \lvert \lambda_i - \mu \rvert \Big\| \sum_{j = 1}^N r_j (\cdot) f_j \Big\|_{\L^2 (0 , 1 ; \L^p (\BR^n ; \BC^n))}.
\end{align}
Notice that $\nu , \sigma > 0$ and that out of symmetry reasons~\eqref{Eq: Quantified approximation gradient} and~\eqref{Eq: Quantified approximation} hold with the symbols $\lambda_j$ and $\mu$ interchanged. This gives~\eqref{Eq: Quantified approximation whole space, general}. Furthermore, if $N = 1$, this gives the $\W_{\loc}^{1 , p}$-convergence properties stated in the lemma. The convergence of $( \nabla^2 u_{\lambda}^R)_{\lambda \in \Sigma_{\theta}}$ follows since the sectoriality of the Stokes operator $A_{p , \BR^n}$, Proposition~\ref{prop-2.5}, implies the Cauchy property in $\L^p (\BR^n ; \BC^{n^3})$ as $\lambda \to 0$, see~\cite[Prop.~2.1.1]{Haase}. All convergences being known, let $\mu \to 0$ in~\eqref{Eq: Quantified approximation gradient} and~\eqref{Eq: Quantified approximation}. This delivers~\eqref{Eq: Quantified approximation whole space, origin}.
\end{proof}

A similar lemma holds on bounded Lipschitz domains. As above, for $\lambda , \mu , \lambda_j \in \Sigma_{\theta}$ and $f , f_j \in \L^p (\BR^n ; \BC^n)$ we write
\begin{align*}
 u_{\lambda}^D := (\lambda + A_{p , D})^{-1} \BP_{p , D} f, \quad  u_{\lambda_j}^D := (\lambda_j + A_{p , D})^{-1} \BP_{p , D} f_j, \quad \text{and} \quad u_{\mu , j}^D := (\mu + A_{p , D})^{-1} \BP_{p , D} f_j
\end{align*}
and denote the associated pressures by $\pi_{\lambda}^D$, $\pi_{\lambda_j}^D$, and $\pi_{\mu , j}^D$.

\begin{lemm}
\label{Lem: Convergence for lambda to zero Lipschitz domain}
Let $\theta \in (0 , \pi)$, $\varepsilon$ and $p$ be subject to Convention~\ref{Conv}, and $f \in \L^p (D ; \BC^n)$. There exists $u_0^D \in \CD(A_{p , D})$ with associated pressure $\pi_0^D \in \L^p_0 (D)$ such that 
\begin{align*}
\left\{ \begin{aligned}
 - \Delta u_0^D + \nabla \pi_0^D &= f && \text{in } D \\
 \dv (u_0^D) &= 0 && \text{in } D \\
 u_0^D &= 0 && \text{on } \partial D
\end{aligned} \right.
\end{align*}
and such that $u_{\lambda}^D \to u_0^D$ in $\W^{1 , p}_{0 , \sigma} (D)$ and in $\W^{2 , p}_{\loc} (D ; \BC^n)$ as $\lambda \to 0$. Furthermore, it holds $\pi_{\lambda}^D \to \pi_0^D$ in $\L^p_0 (D)$ and in $\W^{1 , p}_{\loc} (D)$ as $\lambda \to 0$ with $\lambda \in \Sigma_{\theta}$.\par
Furthermore, there exists $C > 0$ such that for all $N \in \BN$, $f_j \in \L^p (D ; \BC^n)$, and $\lambda_j , \mu \in \Sigma_{\theta} \cup \{ 0 \}$ $(j = 1 , \dots , N)$ it holds
\begin{align}
  \Big\| \sum_{j = 1}^N r_j (\cdot) (u_{\lambda_j}^D - u_{\mu , j}^D) \Big\|_{\L^2 (0 , 1 ; \W^{1 , p}_{0 , \sigma} (D))} \leq C \max_{1 \leq i \leq N} \lvert \mu - \lambda_i \rvert \cdot \Big\| \sum_{j = 1}^N r_j (\cdot) f_j \Big\|_{\L^2 (0 , 1 ; \L^p (D ; \BC^n))} \label{Eq: Quantified approximation Lipschitz domain, velocity} 
\end{align}
and
\begin{align}
\label{Eq: Quantified approximation Lipschitz domain, pressure}
 \Big\| \sum_{j = 1}^N r_j (t) (\pi_{\lambda_j}^D - \pi_{\mu , j}^D) \Big\|_{\L^p_0 (D)} \leq C \max_{1 \leq i \leq N} \lvert \mu - \lambda_i \rvert \cdot \Big\| \sum_{j = 1}^N r_j (\cdot) f_j \Big\|_{\L^2 (0 , 1 ; \L^p (D ; \BC^n))}. 
\end{align}
\end{lemm}

\begin{proof}
Let $N \in \BN$, $f_j \in \L^p (D ; \BC^n)$, and $\lambda_j , \mu \in \Sigma_{\theta}$. The invertibility of $A_{p , D}$ together with the continuous embedding $\CD(A_{p , D}) \subset \W^{1 , p}_{0 , \sigma} (D)$ and the resolvent identity implies 
\begin{align*}
 \Big\| \sum_{j = 1}^N r_j (\cdot) (u_{\lambda_j}^D - u_{\mu , j}^D) \Big\|_{\L^2 (0 , 1 ; \W^{1 , p}_{0 , \sigma} (D))} \leq C \Big\| \sum_{j = 1}^N r_j (\cdot) (\mu - \lambda_j) A_{p , D} (\mu + A_{p , D})^{-1} A_{p , D} u_{\lambda_j}^D \Big\|_{\L^2 (0 , 1 ; \L^p_{\sigma} (D))}.
\end{align*}
An application of Proposition~\ref{prop-3.2} followed by Kahane's contraction principle then yields~\eqref{Eq: Quantified approximation Lipschitz domain, velocity}. \par
Concerning the pressure, use that $A_{p , D} u_{\lambda_j}^D = - \Delta u_{\lambda_j}^D + \nabla \pi_{\lambda_j}^D$ holds in the sense of distributions (the same holds for $\mu$) and estimate by using Bogovski\u{\i}'s operator $\CB$ on $D$ and the resolvent identity for almost every $t \in (0 , 1)$
\begin{align*}
 \Big\| \sum_{j = 1}^N r_j (t) (\pi_{\lambda_j}^D - \pi_{\mu , j}^D) \Big\|_{\L^p_0 (D)} &= \sup_{\substack{h \in \L^{p^{\prime}}_0 (D) \\ \| h \|_{\L^{p^{\prime}} (D)} \leq 1}} \Big\lvert \int_{D} \sum_{j = 1}^N r_j (t) (\pi_{\lambda_j}^D - \pi_{\mu , j}^D) \overline{\dv (\CB h)} \dx \Big\rvert \\
 &\leq \sup_{\substack{h \in \L^{p^{\prime}}_0 (D) \\ \| h \|_{\L^{p^{\prime}} (D)} \leq 1}} \Big\lvert \int_{D} \sum_{j = 1}^N r_j (t) (\mu - \lambda_j) A_{p , D} (\mu + A_{p , D})^{-1} u_{\lambda_j}^D \cdot \overline{\CB h} \dx \Big\rvert \\
 &\qquad + \sup_{\substack{h \in \L^{p^{\prime}}_0 (D) \\ \| h \|_{\L^{p^{\prime}} (D)} \leq 1}} \Big\lvert \int_{D} \sum_{j = 1}^N r_j (t) \nabla (u_{\lambda_j}^D - u_{\mu , j}^D) \cdot \overline{\nabla \CB h} \dx \Big\rvert.
\end{align*}
As a consequence, the boundedness of $\CB : \L^{p^{\prime}}_0 (D) \to \W^{1 , p^{\prime}}_0 (D ; \BC^n)$ together with~\eqref{Eq: Quantified approximation Lipschitz domain, velocity}, Proposition~\ref{prop-3.2}, the invertibility of $A_{p , D}$, and Kahane's contraction principle yields~\eqref{Eq: Quantified approximation Lipschitz domain, pressure}. \par
If $N = 1$,~\eqref{Eq: Quantified approximation Lipschitz domain, velocity} and~\eqref{Eq: Quantified approximation Lipschitz domain, pressure} show that $(u_{\lambda})_{\lambda \in \Sigma_{\theta}}$ is a Cauchy sequence in $\W^{1 , p}_{0 , \sigma} (D)$ and that $(\pi_{\lambda}^D)_{\lambda \in \Sigma_{\theta}}$ is a Cauchy sequence in $\L^p_0 (D)$ as $\lvert \lambda \rvert \to 0$. Moreover, the sectoriality of the Stokes operator $A_{p , D}$, see Proposition~\ref{prop-3.2}, implies that $(A_{p , D} u_{\lambda}^D)_{\lambda \in \Sigma_{\theta}}$ is a Cauchy sequence in $\L^p_{\sigma} (D)$ as well~\cite[Prop.~2.1.1]{Haase}. The closedness of $A_{p , D}$ implies that the limit $u_0^D$ is an element of $\CD(A_{p , D})$. Finally, inner regularity estimates, see~\cite[Thm.~IV.4.1]{G}, imply the convergence of $(u_{\lambda}^D)_{\lambda \in \Sigma_{\theta}}$ to $u_0^D$ in $\W^{2 , q}_{\loc} (D ; \BC^n)$ and of $(\pi_{\lambda}^D)_{\lambda \in \Sigma_{\theta}}$ to $\pi_0^D$ in $\W^{1 , q}_{\loc} (D)$.
\end{proof}

\begin{prop}
\label{Prop: The operator T_0}
Let $\Omega$ be an exterior Lipschitz domain in $\BR^n$ and $\theta \in (0 , \pi)$. Let $\varepsilon$ and $p$ be subject to Convention~\ref{Conv} with $p < n / 2$. Then there exists a compact operator $T_0 \in \CL(\L^p (\Omega ; \BC^n))$ that satisfies $\supp (T_0 f) \subset \overline{K_1}$ for all $f \in \L^p (\Omega ; \BC^n)$ and for all $\gamma > 0$ and $\mu \in \Sigma_{\theta} \cup \{ 0 \}$ there exists $\delta > 0$ such that
\begin{align}
\label{Eq: Definition of T_0}
 \CR_{\L^p (\Omega ; \BC^n)} \big\{ T_{\lambda} - T_{\mu} \mid \lambda \in \Sigma_{\theta} \cap B_{\delta} (\mu) \big\} \leq \gamma.
\end{align}
Moreover, it holds $T_0 \in \CL (\L^p (\Omega ; \BC^n) , \W^{1 , p} (\Omega ; \BC^n))$ and $T_0$ is consistent in the $\L^p$-scale, \textit{i.e.}, if $f \in \L^{p_1} (\Omega ; \BC^n) \cap \L^{p_2} (\Omega ; \BC^n)$ with $p_1 , p_2 < n / 2$ subject to Convention~\ref{Conv}, then the limits $\lim_{\lvert \lambda \rvert \to 0 , \lambda \in \Sigma_{\theta}} T_{\lambda} f$ taken with respect to $\L^{p_1} (\Omega ; \BC^n)$ and $\L^{p_2} (\Omega ; \BC^n)$ coincide.
\end{prop}

\begin{proof}
We concentrate mainly on the case $\mu = 0$. Notice that the compactness of $T_0$ and $\supp (T_0 f) \subset \overline{K_1}$ for all $f \in \L^p (\Omega ; \BC^n)$ will be direct consequences of the convergence in~\eqref{Eq: Definition of T_0} and Lemma~\ref{Lem: Properties of T_lambda}. To establish~\eqref{Eq: Definition of T_0}, define for $f \in \L^p (\Omega ; \BC^n)$, $\lambda \in \Sigma_{\theta}$, and $u_{\lambda}^R$, $u_{\lambda}^D$, and $\pi_{\lambda}^D$ subject to~\eqref{eq-4.1} and~\eqref{eq-4.2}
\begin{align*}
 S_{\lambda}^1 f &:= [(\nabla \varphi) \cdot \nabla] u_{\lambda}^R, &&S_{\lambda}^2 f := [(\nabla \varphi) \cdot \nabla] u_{\lambda}^D, &&S_{\lambda}^3 f := (\Delta \varphi) u_{\lambda}^R \\
 S_{\lambda}^4 f &:= (\Delta \varphi) u_{\lambda}^D, &&S_{\lambda}^5 f := (\nabla \varphi) \pi_{\lambda}^D,  &&S_{\lambda}^6 f := \lambda \CB_1 \big( (\nabla \varphi) \cdot u_{\lambda}^R \big) \\
 S_{\lambda}^7 f &:= \Delta \CB_1 \big( (\nabla \varphi) \cdot u_{\lambda}^R \big), &&S_{\lambda}^8 f := \lambda \CB_1 \big( (\nabla \varphi) \cdot u_{\lambda}^D \big), &&S_{\lambda}^9 f := \Delta \CB_1 \big( (\nabla \varphi) \cdot u_{\lambda}^D \big).
\end{align*}
If $g$ is given by $(\I - \BP_{p , \BR^n}) f^R = \nabla g$, then~\eqref{4.6} delivers the relation
\begin{align*}
 T_{\lambda} f = - 2 S_{\lambda}^1 f + 2 S_{\lambda}^2 f - S_{\lambda}^3 f + S_{\lambda}^4 f - S_{\lambda}^5 f + (\nabla \varphi) g - S_{\lambda}^6 f + S_{\lambda}^7 f + S_{\lambda}^8 f - S_{\lambda}^9 f.
\end{align*}
First of all, notice that $g$ does not depend on $\lambda$ so that this term does not have to be investigated. Let $\mu \in \Sigma_{\theta}$. Then, since $\supp (\nabla \varphi) \subset \overline{B_{R + 2}} \setminus B_{R + 1}$ is compact, the convergences and estimates proven in Lemmas~\ref{Lem: Convergences for lambda to zero} and~\ref{Lem: Convergence for lambda to zero Lipschitz domain} carry over to respective convergences and estimates of $S_{\lambda}^1$, $S_{\lambda}^2$, $S_{\lambda}^3$, $S_{\lambda}^4$, and $S_{\lambda}^5$. Analogously, taking Proposition~\ref{prop-3.4} into account, convergences and estimates of $S_{\lambda}^7$ and $S_{\lambda}^9$ follow as well. Finally, to estimate $S_{\lambda}^6$ the triangle inequality gives for $\mu \in \Sigma_{\theta} \cap B_{\delta} (0)$
\begin{align*}
 \CR_{\L^p (\Omega ; \BC^n)} \big\{ S_{\lambda}^6 \mid \lambda \in \Sigma_{\theta} \cap B_{\delta} (0) \big\} \leq \CR_{\L^p (\Omega ; \BC^n)} \big\{ S_{\lambda}^6 - S_{\mu}^6 \mid \lambda \in \Sigma_{\theta} \cap B_{\delta} (0) \big\} + \| S_{\mu}^6 \|_{\CL(\L^p (\Omega ; \BC^n))}.
\end{align*}
Employing~\eqref{Eq: Quantified approximation whole space, general} and Kahane's contraction principle, the first term on the right-hand side is small whenever $\delta$ is small. For the second term on the right-hand side let $q > p$ with $1 / p - 1 / q < 2 / n$, use the boundedness of the Bogovski\u{\i} operator followed by H\"older's inequality and Lemma~\ref{Lem: R-bounded L^p-L^q estimates whole space} to estimate
\begin{align*}
 \| S_{\mu}^6 f \|_{\L^p (\Omega ; \BC^n)} \leq C \lvert \mu \rvert \| u_{\mu}^R \|_{\L^q (\BR^n ; \BC^n)} \leq \lvert \mu \rvert^{\frac{n}{2} (\frac{1}{p} - \frac{1}{q})} \| f \|_{\L^p (\Omega ; \BC^n)}.
\end{align*}
Analogously, one estimates $S_{\lambda}^8$. It follows that the $\CR$-norms of $(S_{\lambda}^6)_{\lambda \in \Sigma_{\theta} \cap B_{\delta} (0)}$ and $(S_{\lambda}^8)_{\lambda \in \Sigma_{\theta} \cap B_{\delta} (0)}$ are small, whenever $\delta$ is small. In particular, $S_{\lambda}^6 f$ and $S_{\lambda}^8 f$ converge to zero as $\lambda \to 0$. This establishes the existence of $T_0$. \par
To show that $T_0$ maps boundedly into $\W^{1 , p} (\Omega ; \BC^n)$, notice that this is true for each $T_{\lambda}$ if $\lambda \neq 0$ by Lemma~\ref{Lem: Properties of T_lambda}. Now, $T_\lambda f$ converges to $T_0 f$ in $\W^{1 , p} (\Omega ; \BC^n)$ as $\lambda \to 0$ due to Lemmas~\ref{Lem: Convergences for lambda to zero} and~\ref{Lem: Convergence for lambda to zero Lipschitz domain}. By the Banach--Steinhaus theorem, we find $T_0 \in \CL (\L^p (\Omega ; \BC^n) , \W^{1 , p} (\Omega ; \BC^n))$. \par
The case $\mu \neq 0$ follows literally by the same reasoning.
\end{proof}

\subsection{Step~4: Invertibility of $\I + T_{\lambda}$}
A direct consequence of Lemma~\ref{Lem: Properties of T_lambda} and Proposition~\ref{Prop: The operator T_0} is that for $\lambda \in \Sigma_{\theta} \cup \{ 0 \}$ the operator $\I + T_{\lambda}$ is Fredholm and thus the Fredholm alternative reduces the question of the invertibility of $\I + T_{\lambda}$ to the injectivity of $\I + T_{\lambda}$.

\begin{prop}
\label{Prop: Injectivity of I + T_lambda}
Let $\Omega \subset \BR^n$ be an exterior Lipschitz domain and $\theta \in (0 , \pi)$. Let $\varepsilon$ and $p$ be subject to Convention~\ref{Conv}. Then for every $\lambda \in \Sigma_{\theta}$ the operator $\I + T_{\lambda} : \L^p (\Omega ; \BC^n) \to \L^p (\Omega ; \BC^n)$ is injective. If additionally $p < n / 2$, the operator $\I + T_0 : \L^p (\Omega ; \BC^n) \to \L^p (\Omega ; \BC^n)$ is injective.
\end{prop}

\begin{proof}
Let $\lambda \in \Sigma_{\theta} \cup \{ 0 \}$ and assume that there exists $f \in \L^p (\Omega ; \BC^n)$ with $(\I + T_{\lambda}) f = 0$ (with $p < n / 2$ in the case $\lambda = 0$). In other words, it holds
\begin{align*}
 f = - T_{\lambda} f \qquad \text{in } \Omega.
\end{align*}
As a consequence of Lemma~\ref{Lem: Properties of T_lambda} and Proposition~\ref{Prop: The operator T_0} the function $f$ satisfies
\begin{align*}
 \supp (f) \subset \overline{K_1}.
\end{align*}
On the one hand, this support property of $f$ ensures that $f \in \L^q (\Omega ; \BC^n)$ for all $1 \leq q \leq p$. On the other hand, Lemma~\ref{Lem: Properties of T_lambda} and Proposition~\ref{Prop: The operator T_0} ensure that $f \in \W^{1 , p} (\Omega ; \BC^n)$. Thus, Sobolev's embedding theorem entails $f \in \L^{p^*} (\Omega ; \BC^n)$ with $p^* := n p / (n - p)$. If $\lambda = 0$ and $p^* < n / 2$, then Proposition~\ref{Prop: The operator T_0} ensures that $f \in \W^{1 , p^*} (\Omega ; \BC^n)$ which is embedded by Sobolev's embedding theorem into $\L^{p^{**}} (\Omega ; \BC^n)$ with $p^{**} := n p^* / (n - p^*)$. If $\lambda \neq 0$ and $p^* < n$, then Lemma~\ref{Lem: Properties of T_lambda} together with Sobolev's embedding theorem implies $f \in \L^{p^{**}} (\Omega ; \BC^n)$ as well. Iterate this procedure until $f \in \L^q (\Omega ; \BC^n)$ for each $1 \leq q < n$ (if $\lambda = 0$) or $f \in \L^q (\Omega ; \BC^n)$ for each $1 \leq q < \infty$ (if $\lambda \neq 0$) is established. \par
Let $\lambda \neq 0$. We find in particular $f \in \L^2 (\Omega ; \BC^n)$ and thus $U_{\lambda} f \in \W^{1 , 2}_{0 , \sigma} (\Omega)$ and $\Pi_{\lambda} f \in \L^2_{\loc} (\Omega)$. Consequently,
\begin{align*}
 \lambda \int_{\Omega} \lvert U_{\lambda} f \rvert^2 \dx + \int_{\Omega} \lvert \nabla U_{\lambda} f \rvert^2 \dx = \int_{\Omega} (\I + T_{\lambda}) f \cdot \overline{U_{\lambda} f} \dx = 0.
\end{align*}
Since $\lambda \in \Sigma_{\theta}$ it follows that
\begin{align*}
 \int_{\Omega} \lvert U_{\lambda} f \rvert^2 \dx = 0 \qquad \text{and} \qquad \Pi_{\lambda} f = c \quad \text{in } \Omega \text{ for some } c \in \BC.
\end{align*}
\indent Let $\lambda = 0$. Since in particular $f \in \L^q (\Omega ; \BC^n)$ for all $1 \leq q < n / 2$ subject to Convention~\ref{Conv}, Proposition~\ref{Prop: The operator T_0} ensures for these $q$ the $\L^q$-convergence
\begin{align*}
 (\I + T_0) f = \lim_{\substack{ \lvert \mu \rvert \to 0 \\ \mu \in \Sigma_{\theta} }} (\I + T_{\mu}) f = \lim_{\substack{ \lvert \mu \rvert \to 0 \\ \mu \in \Sigma_{\theta} }} \big[ (\mu - \Delta) U_{\mu} f + \nabla \Pi_{\mu} f \big].
\end{align*}
In particular, this convergence is valid for some $q$ satisfying $2 n / (n + 2) < q < n / 2$. Moreover, since also $f \in \L^r (\Omega ; \BC^n)$ for all $1 \leq r < n$, we find $f \in \L^2(\Omega ; \BC^n)$ and thus that for each $\mu \in \Sigma_{\theta}$ it holds
\begin{align*}
  \mu \int_{\Omega} \lvert U_{\mu} f \rvert^2 \dx + \int_{\Omega} \lvert \nabla U_{\mu} f \rvert^2 \dx = \int_{\Omega} (\I + T_{\mu}) f \cdot \overline{U_{\mu} f} \dx.
\end{align*}
By assumption $(\I + T_{\mu}) f$ converges to zero in $\L^q (\Omega ; \BC^n)$ and notice that the support of $(\I + T_{\mu}) f$ is contained in $\overline{K_1}$. Moreover, by Sobolev's embedding theorem and the special choice of $q$ we have $\W^{2 , q}_{\loc} (\Omega ; \BC^n) \subset \L^{q^{\prime}}_{\loc} (\Omega ; \BC^n)$. Thus, $(U_{\mu})_{\mu \in \Sigma_{\theta}}$ is bounded in $\L^{q^{\prime}} (K_1 ; \BC^n)$ as $\lvert \mu \rvert \to 0$ by Lemmas~\ref{Lem: Convergences for lambda to zero} and~\ref{Lem: Convergence for lambda to zero Lipschitz domain}. It follows that
\begin{align*}
 \int_{\Omega} \lvert \nabla U_0 f \rvert^2 \dx \leq C \lim_{\substack{\lvert \mu \rvert \to 0 \\ \mu \in \Sigma_{\theta}}} \bigg\lvert \mu \int_{\Omega} \lvert U_{\mu} f \rvert^2 \dx + \int_{\Omega} \lvert \nabla U_{\mu} f \rvert^2 \dx \bigg\rvert = 0.
\end{align*}
Consequently, in all cases ($\lambda = 0$ and $\lambda \neq 0$) it holds $U_{\lambda} f = 0$ and $\Pi_{\lambda} f = c$ for some $c \in \BC$. \par
Combining~\eqref{Eq: The approximate solution} and the definition of $\Pi_{\lambda} f$ above~\eqref{eq-4.5} together with Lemmas~\ref{Lem: Convergences for lambda to zero} and~\ref{Lem: Convergence for lambda to zero Lipschitz domain}, we find that in both cases ($\lambda = 0$ and $\lambda \neq 0$) $U_{\lambda} f$ and $\Pi_{\lambda} f$ are given by
\begin{align}
\label{Eq: Formula above integrabilities}
 U_{\lambda} f = \varphi u_{\lambda}^R + (1 - \varphi) u_{\lambda}^D - \CB_1 \big( (\nabla \varphi) \cdot (u_{\lambda}^R - u_{\lambda}^D) \big) \quad \text{and} \quad \Pi_{\lambda} f = (1 - \varphi) \pi_{\lambda}^D + \varphi g.
\end{align}
Since in all cases $f \in \L^q (\Omega ; \BC^n)$ for all $1 \leq q < n$, Lemma~\ref{Lem: Convergences for lambda to zero} asserts that for all $1 < r < n / 2$ its holds $u_{\lambda}^R \in \W^{2 , r}_{\loc} (\BR^n ; \BC^n)$ with $\nabla^2 u_{\lambda}^R \in \L^r (\BR^n ; \BC^{n^3})$ and $\dv (u_{\lambda}^R) = 0$ and $g \in \L^r_{\loc} (\BR^n)$ with $\nabla g \in \L^r (\BR^n ; \BC^n)$. Furthermore, concerning $u_{\lambda}^D$ and $\pi_{\lambda}^D$, Lemma~\ref{Lem: Convergence for lambda to zero Lipschitz domain} asserts that for all $r < n$ subject to Convention~\ref{Conv}, $u_{\lambda}^D \in \W^{1 , r}_{0 , \sigma} (D) \cap \W^{2 , r}_{\loc} (D ; \BC^n)$ and $\pi_{\lambda}^D \in \L^r_0 (D) \cap \W^{1 , r}_{\loc} (D)$. By the definition of $\varphi$, the fact that $\CB_1 ((\nabla \varphi) \cdot (u_{\lambda}^R - u_{\lambda}^D))$ is extended by zero to all of $\BR^n$ outside of $K_1$, and the fact that $U_{\lambda} f = 0$ in $\Omega$, we find
\begin{align*}
 u_{\lambda}^R (x) = 0 \quad \text{for } \lvert x \rvert > R + 3 \qquad \text{and} \qquad u_{\lambda}^D (x) = 0 \quad \text{for } x \in \Omega \cap B_R (0).
\end{align*}
Furthermore, we also find $\pi_{\lambda}^D$ to be constant on $\Omega \cap B_R (0)$ and $g$ to be constant on $B_{R + 3} (0)^c$. Thus, $u_{\lambda}^D$ and $\pi_{\lambda}^D$ can be extended constantly to functions in $\W^{1 , 2}_{0 , \sigma} (B_{R + 5} (0))$ and $\L^2 (B_{R + 5} (0))$, respectively. Let us denote these functions by $u_{\lambda}^D$ and $\pi_{\lambda}^D$ again. \par
Next, recall the definition of $\eta$ and notice that the fact $\supp(f) \subset \overline{K_1}$ implies
\begin{align*}
 \nabla \eta \cdot f = 0 \quad \text{in } K_2 \qquad \text{and} \qquad \eta f  = f \quad \text{in } \Omega.
\end{align*}
Consequently, by definition of $f^D$ it holds $f^D = f$ in $D$. Since $2 n / (n + 2) < n / 2$, the integrabilities stated below~\eqref{Eq: Formula above integrabilities} imply $u_{\lambda}^R \in \W^{2 , 2n / (n + 2)}_{0 , \sigma} (B_{R + 5} (0))$ and by Sobolev's embedding theorem that $u_{\lambda}^R \in \W^{1 , 2}_{0 , \sigma} (B_{R + 5} (0))$. Consequently, $u_{\lambda}^D$ and $u_{\lambda}^R$ solve the Stokes (resolvent) problem in $B_{R + 5} (0)$ subject to homogeneous Dirichlet boundary conditions and the same right-hand side $f$. Consequently, these functions have to coincide in $B_{R + 5} (0)$ and there exists a constant $c_1 \in \BC$ such that $\pi_{\lambda}^D = g + c_1$ in $B_{R + 5} (0)$. Furthermore, it follows that $u_{\lambda}^R = U_{\lambda} f$ and thus
\begin{align*}
 u_{\lambda}^R = 0 \quad \text{in } \Omega,
\end{align*}
hence $\BP_{p , \BR^n} f = 0$. Since $\pi_{\lambda}^D$ and $g$ are normalized to have average zero on $D$, see~\eqref{Eq: Normalization of g}, it follows that $c_1 = 0$. Since $\Pi_{\lambda} f = c_2$ for some constant $c_2 \in \BC$ it follows that
\begin{align*}
 c_2 = (1 - \varphi) \pi_{\lambda}^D + \varphi g = (1 - \varphi) \pi_{\lambda}^D + \varphi \pi_{\lambda}^D = \pi_{\lambda}^D,
\end{align*}
and again, since the average of $\pi_{\lambda}^D$ is zero in $D$, it follows that $c_2 = 0$. Consequently, $g$ vanishes on all of $\BR^n$, which implies $f^R = 0$ and thus $f = 0$.
\end{proof}

\begin{lemm}
\label{Lem: Control for small lambda}
Let $\varepsilon$ and $p$ be subject to Convention~\ref{Conv} with $p < n / 2$, $\theta \in (0 , \pi)$, and $\lambda^* > 0$. Then for all $\lambda \in \overline{\Sigma_{\theta} \cap B_{\lambda^*} (0)}$ the operator $\I + T_{\lambda} : \L^p (\Omega ; \BC^n) \to \L^p (\Omega ; \BC^n)$ is invertible and there exists a constant $C > 0$ such that
\begin{align*}
 \CR_{\L^p (\Omega ; \BC^n)} \big\{ (\I + T_{\lambda})^{-1} \mid \lambda \in \Sigma_{\theta} \cap B_{\lambda^*} (0) \big\} \leq C.
\end{align*} 
\end{lemm}

\begin{proof}
By Lemma~\ref{Lem: Properties of T_lambda} and Propositions~\ref{Prop: The operator T_0} and~\ref{Prop: Injectivity of I + T_lambda} for each $\mu \in \overline{\Sigma_{\theta} \cap B_{\lambda^*} (0)}$ the operators $\I + T_{\mu}$ are invertible. Moreover, by Proposition~\ref{Prop: The operator T_0} there exists $\delta_{\mu} > 0$ such that
\begin{align*}
 \CR_{\L^p (\Omega ; \BC^n)} \big\{ T_{\lambda} - T_{\mu} \mid \lambda \in \Sigma_\theta \cap B_{\delta_{\mu}} (\mu) \big\} \leq \big(2 \| (\I + T_{\mu})^{-1} \|_{\CL(\L^p (\Omega ; \BC^n))} \big)^{-1}.
\end{align*}
Since
\begin{align*}
 \big\{ B_{\delta_{\mu}} (\mu) \mid \mu \in \overline{\Sigma_{\theta} \cap B_{\lambda^*} (0)} \big\}
\end{align*}
is an open covering of the compact set $\overline{\Sigma_{\theta} \cap B_{\lambda^*} (0)}$, there exists $m \in \BN$ together with $\mu_j \in \overline{\Sigma_{\theta} \cap B_{\lambda^*} (0)}$ $(j = 1 , \dots , m)$, such that
\begin{align*}
 \Sigma_{\theta} \cap B_{\lambda^*} (0) \subset \bigcup_{j = 1}^m B_{\delta_{\mu_j}} (\mu_j).
\end{align*}
By the choices above, a usual Neumann series argument shows that for all $j = 1 , \dots , m$ it holds
\begin{align*}
 \CR_{\L^p (\Omega ; \BC^n)} \big\{ (\I + T_{\lambda})^{-1} \mid \lambda \in \Sigma_\theta \cap B_{\delta_{\mu_j}} (\mu_j) \big\} < 2 \| (\I + T_{\mu_j})^{-1} \|_{\CL(\L^p (\Omega ; \BC^n))}.
\end{align*}
Thus, the lemma is proved.
\end{proof}

\subsection*{Step~5: Proof of Theorems~\ref{th-1.1} and~\ref{Thm: Maximal regularity}}

Let $\varepsilon$ and $p$ be subject to Convention~\ref{Conv} and $f \in \L^p (\Omega ; \BC^n)$.
Combining Lemma~\ref{Lem: Properties of T_lambda} and Proposition~\ref{Prop: Injectivity of I + T_lambda}, we infer that $\I + T_1$ is invertible. Thus, defining
\begin{align}
\label{Eq: -1 in resolvent set}
 u := U_{1} (\I + T_{1})^{-1} f \qquad \text{and} \qquad \pi := \Pi_{1} (\I + T_{1})^{-1} f,
\end{align}
we find by~\eqref{Eq: Ineq 2} that there exists $C > 0$ such that
\begin{align}
\label{Eq: W1p estimate}
 \| u \|_{\W^{1 , p}_{0 , \sigma} (\Omega)} \leq C \| f \|_{\L^p_{\sigma} (\Omega)}.
\end{align}
We argue now, that $-1 \in \rho(A_p)$. Let $p = 2$ for a moment. Notice that since the operator $U_1$ and $\Pi_1$ solve~\eqref{eq-4.5} in the sense of distributions, moreover, since~\eqref{Eq: W1p estimate} holds true, and since $A_2$ is defined by a sesquilinear form, we find
\begin{align}
\label{Eq: L2 resolvent}
 U_1 (\I + T_1)^{-1} f = (\I + A_2)^{-1} f \qquad (f \in \L^2_{\sigma} (\Omega)).
\end{align}
Now, consider the case of general $p$ and recall the definition of $A_p$ given in the introduction. Let $f \in \L^p_{\sigma} (\Omega)$. Then $f$ can be approximated in $\L^p_{\sigma} (\Omega)$ by a sequence $(f_k)_{k \in \BN} \subset \C_{c , \sigma}^{\infty} (\Omega) \subset \L^2_{\sigma}(\Omega) \cap \L^p_{\sigma} (\Omega)$. Let $u_k$ be given by~\eqref{Eq: -1 in resolvent set} but with right-hand side $f_k$, so that $u_k \in \CD(A_2)$. The estimate~\eqref{Eq: W1p estimate} especially implies that $u_k \in \L^p_{\sigma} (\Omega)$, which shows that 
\begin{align*}
 u_k \in \{ v \in \CD (A_2) \cap \L^p_{\sigma} (\Omega) : A_2 v \in \L^p_{\sigma} (\Omega) \} = \CD(A_2|_{\L^p_{\sigma}}).
\end{align*}
By~\eqref{Eq: W1p estimate}, the sequence $(u_k)_{k \in \BN}$ converges in $\W^{1 , p}_{0 , \sigma} (\Omega)$ to $u$ defined by~\eqref{Eq: -1 in resolvent set} with right-hand side $f$. Since $A_p$ is the closure of $A_2|_{\L^p_{\sigma}}$ in $\L^p_{\sigma} (\Omega)$, we find $u \in \CD(A_p)$ and $u + A_p u = f$. This proves the surjectivity. \par
For the injectivity, notice that by~\eqref{Eq: L2 resolvent} it holds
\begin{align*}
 U_1 (\I + T_1)^{-1} (\I + A_2|_{\L^p_{\sigma}}) u = u \qquad (u \in \CD(A_2|_{\L^p_{\sigma}})).
\end{align*}
Since $A_p$ is the closure of $A_2|_{\L^p_{\sigma}}$ in $\L^p_{\sigma} (\Omega)$, this identity carries over to all $u \in \CD(A_p)$ by taking limits while using~\eqref{Eq: W1p estimate}. This proves the injectivity. Since $A_p$ is by definition closed, it follows $-1 \in \rho(A_p)$ and~\eqref{Eq: W1p estimate} implies the continuous inclusion $\CD(A_p) \subset \W^{1 , p}_{0 , \sigma} (\Omega)$. Notice that with the same reasoning, one readily verifies that $\Sigma_{\theta} \subset \rho(- A_p)$ holds for every $\theta \in (0 , \pi)$. We continue by estimating the resolvent. \par
Next, let $p < n / 2$, $f \in \L^p_{\sigma} (\Omega)$, and $\lambda \in \Sigma_{\theta}$. By construction, $U_{\lambda} f$ and $\Pi_{\lambda} f$ solve~\eqref{eq-4.5}. Decompose by means of the Helmholtz decomposition~\eqref{Eq: Helmholtz general} and Proposition~\ref{Prop: Helmholtz on exterior}
\begin{align*}
 f + T_{\lambda} f = f + \BP_{p , \Omega} T_{\lambda} f + (\I - \BP_{p , \Omega}) T_{\lambda} f =: f + \BP_{p , \Omega} T_{\lambda} f + \nabla \Phi_{\lambda} f.
\end{align*}
Thus, $U_{\lambda} f$ and $\Pi_{\lambda} f - \Phi_{\lambda} f$ solve~\eqref{Eq: Stokes resolvent problem} with right-hand side $f + \BP_{p , \Omega} T_{\lambda} f$.
Let $\lambda^* \geq 1$ be the number obtained in Corollary~\ref{Cor: Control for large lambda}, \textit{i.e.}, $\lambda^*$ is chosen such that
\begin{align*}
 \CR_{\L^p_{\sigma} (\Omega)} \big\{ (\I + \BP_{p , \Omega} T_{\lambda})^{-1} \mid \lambda \in \Sigma_\theta , \lvert \lambda \rvert \geq \lambda^* \big\} \leq 2.
\end{align*}
Thus, if $\lvert \lambda \rvert \geq \lambda^*$ the functions
\begin{align*}
 u := U_{\lambda} (\I + \BP_{p , \Omega} T_{\lambda})^{-1} f \qquad \text{and} \qquad \pi := \Pi_{\lambda} (\I + \BP_{p , \Omega} T_{\lambda})^{-1} f  - \Phi_{\lambda} (\I + \BP_{p , \Omega} T_{\lambda})^{-1} f 
\end{align*}
solve the Stokes resolvent problem~\eqref{Eq: Stokes resolvent problem} with right-hand side $f$. Since compositions of $\CR$-bounded sets are $\CR$-bounded, Lemma~\ref{Lem: Perturbed solution operator} together with Corollary~\ref{Cor: Control for large lambda} imply that there exists $C > 0$ such that
\begin{align*}
 \CR_{\L^p_{\sigma} (\Omega)} \big\{ \lambda (\lambda + A_p)^{-1} \mid \lambda \in \Sigma_{\theta} , \lvert \lambda \rvert \geq \lambda^* \big\} \leq C.
\end{align*}
Furthermore, the boundedness of $\BP_{p , \Omega}$, see Proposition~\ref{Prop: Helmholtz on exterior}, implies that
\begin{align}
\label{Eq: Maximal regularity large lambda}
 \CR_{\L^p (\Omega ; \BC^n)} \big\{ \lambda (\lambda + A_p)^{-1} \BP_{p , \Omega} \mid \lambda \in \Sigma_{\theta} , \lvert \lambda \rvert \geq \lambda^* \big\} \leq C
\end{align}
for some possibly different constant $C > 0$. \par
If $\lvert \lambda \rvert < \lambda^*$, then Lemma~\ref{Lem: Control for small lambda} allows us to conclude that the solutions to~\eqref{Eq: Stokes resolvent problem} with right-hand side $f$ are given by
\begin{align*}
 u := U_{\lambda} (\I + T_{\lambda})^{-1} f \qquad \text{and} \qquad \pi := \Pi_{\lambda} (\I + T_{\lambda})^{-1} f.
\end{align*}
Combining the same lemma with Lemma~\ref{Lem: Perturbed solution operator} yields the existence of a constant $C > 0$ such that
\begin{align}
\label{Eq: Maximal regularity small lambda}
 \CR_{\L^p (\Omega ; \BC^n)} \big\{ \lambda (\lambda + A_p)^{-1} \BP_{p , \Omega} \mid \lambda \in \Sigma_{\theta} , \lvert \lambda \rvert < \lambda^* \big\} \leq C.
\end{align}
Since the union of two $\CR$-bounded sets is again $\CR$-bounded (this follows by an application of Kahane's contraction principle) it follows by~\eqref{Eq: Maximal regularity large lambda} and~\eqref{Eq: Maximal regularity small lambda} that
\begin{align}
\label{Eq: Maximal regularity small p}
 \CR_{\L^p (\Omega ; \BC^n)} \big\{ \lambda (\lambda + A_p)^{-1} \BP_{p , \Omega} \mid \lambda \in \Sigma_{\theta} \big\} \leq C
\end{align}
for some constant $C > 0$. Since in the case $p = 2$ uniform boundedness implies $\CR$-boundedness, see Remark~\ref{Rem: R-boundedness implies uniform boundedness},~\eqref{Eq: Maximal regularity small p} holds also true in the case $p = 2$. By complex interpolation, it follows that~\eqref{Eq: Maximal regularity small p} holds true for all $p \leq 2$ subject to Convention~\ref{Conv}. Finally, the duality result in~\cite[Lem.~3.1]{Kalton_Weis} implies the validity of~\eqref{Eq: Maximal regularity small p} for all $p$ that satisfy Convention~\ref{Conv}. Now, Remark~\ref{Rem: R-boundedness implies uniform boundedness} implies that $- A_p$ generates a bounded analytic semigroup $(T(t))_{t \geq 0}$ on $\L^p_{\sigma} (\Omega)$ and Proposition~\ref{Prop: Weis} implies Theorem~\ref{Thm: Maximal regularity}.  \par
In order to prove Theorem~\ref{th-1.1}, notice that $\CR$-boundedness implies uniform boundedness of the particular family of operators. Thus, proceeding as in the proof of Theorem~\ref{Thm: Maximal regularity}, we conclude by Lemma~\ref{Lem: Perturbed solution operator} that for $p < n / 2$ and for all $q \geq p$ with $\sigma := n (1 / p - 1 / q) / 2 \leq 1 / 2$ there exists a constant $C > 0$ such that for all $\lambda \in \Sigma_{\theta}$ it holds
\begin{align}
\label{Eq: Endpoint Lp-Lq for resolvent}
 \lvert \lambda \rvert^{1 - \sigma} \| (\lambda + A_p)^{-1} \BP_{p , \Omega} \|_{\CL (\L^p (\Omega ; \BC^n) , \L^q (\Omega ; \BC^n))} \leq C.
\end{align}
Moreover, if additionally $\sigma < 1 / 2$, there exists $C > 0$ such that
\begin{align}
\label{Eq: Endpoint Lp-Lq for gradient estimates}
 \lvert \lambda \rvert^{1 / 2 - \sigma} \| \nabla (\lambda + A_p)^{-1} \BP_{p , \Omega} \|_{\CL (\L^p (\Omega ; \BC^n) , \L^q (\Omega ; \BC^{n^2}))} \leq C.
\end{align}
Complex interpolation between~\eqref{Eq: Maximal regularity small p} and~\eqref{Eq: Endpoint Lp-Lq for resolvent} yields the validity of~\eqref{Eq: Endpoint Lp-Lq for resolvent} for all $p \leq q < n$ that both satisfy Convention~\ref{Conv} and $\sigma \leq 1 / 2$. Furthermore, complex interpolation of~\eqref{Eq: Endpoint Lp-Lq for gradient estimates} with
\begin{align*}
 \lvert \lambda \rvert^{1 / 2} \| \nabla (\lambda + A_2)^{-1} \BP_{2 , \Omega} \|_{\CL (\L^2 (\Omega ; \BC^n) , \L^2 (\Omega ; \BC^{n^2}))} \leq C,
\end{align*}
(which follows as usual by testing the resolvent equation with the solution $u$), implies~\eqref{Eq: Endpoint Lp-Lq for gradient estimates} for all $p \leq q < n$ that satisfy $p \leq 2$, $\sigma < 1 / 2$, and Convention~\ref{Conv}. \par
Next, employing Proposition~\ref{Prop: Transference} yields that for all $p \leq q < n$ that both satisfy Convention~\ref{Conv} and $\sigma \leq 1 / 2$ there exists a constant $C > 0$ such that for all $t > 0$ it holds
\begin{align}
\label{Eq: Lp-Lq for the semigroup restricted}
 t^{\sigma} \| T(t) \BP_{p , \Omega} \|_{\CL (\L^p (\Omega ; \BC^n) , \L^q (\Omega ; \BC^n))} \leq C.
\end{align}
To get rid of the condition $\sigma \leq 1 / 2$, employ for some suitable $k \in \BN$ the semigroup law $T(t) = T(t / k)^k$ and use~\eqref{Eq: Lp-Lq for the semigroup restricted} $k$ times in a row. This implies the validity of~\eqref{Eq: Lp-Lq estimates} but only if the additional condition $q < n$ is satisfied. \par
Concerning the $\L^p$-$\L^q$-estimates for the gradient of the Stokes semigroup, Proposition~\ref{Prop: Transference} implies that for all $p \leq q < n$ that satisfy $p \leq 2$, $\sigma < 1 / 2$, and Convention~\ref{Conv} that there exists $C > 0$ such that for all $t > 0$ it holds
\begin{align}
\label{Eq: Lp-Lq gradient estimates restricted}
 t^{\sigma + 1 / 2} \| \nabla T(t) \BP_{p , \Omega} \|_{\CL (\L^p (\Omega ; \BC^n) , \L^q (\Omega ; \BC^{n^2}))} \leq C.
\end{align}
To get rid of the condition $\sigma < 1 / 2$, employ the semigroup law $T(t) = T(t / 2) T(t / 2)$ and use first~\eqref{Eq: Lp-Lq gradient estimates restricted} and then~\eqref{Eq: Lp-Lq estimates}. This implies the validity of~\eqref{Eq: Lp-Lq gradient estimates}. \par
Finally, we combine~\eqref{Eq: Lp-Lq for the semigroup restricted} and~\eqref{Eq: Lp-Lq gradient estimates restricted} in order to deduce~\eqref{Eq: Lp-Lq estimates} for the whole range of numbers $p$ and $q$. Indeed, let first $p = 2$ and $q \geq n$ satisfying Convention~\ref{Conv} (we only proceed, if such a number $q$ exists, if not, then the proof is already finished). Let $f \in \L^2 (\Omega ; \BC^n)$ and $2 \leq r < n$, $\alpha \in [0 , 1]$ with $1 / r - 1 / q = \alpha / n$. Then, by the Gagliardo--Nirenberg inequality, it holds
\begin{align}
\label{Eq: Final interpolation}
 \| T(t) \BP_{2 , \Omega} f \|_{\L^q (\Omega ; \BC^n)} &\leq C \| \nabla T(t) \BP_{2 , \Omega} f \|_{\L^r (\Omega ; \BC^n)}^{\alpha} \| T(t) \BP_{2 , \Omega} f \|_{\L^r (\Omega ; \BC^n)}^{1 - \alpha} \\
 &\leq C t^{- \frac{\alpha}{2} - \frac{n}{2} (\frac{1}{2} - \frac{1}{r})} \| f \|_{\L^2 (\Omega ; \BC^n)}.
\end{align}
Notice that $\alpha / 2 + n (1 / 2 - 1 / r) / 2 = n (1 / 2 - 1 / q) / 2$. Performing another complex interpolation between~\eqref{Eq: Final interpolation} and the uniform estimate
\begin{align*}
 \| T(t) \BP_{q , \Omega} \|_{\CL(\L^q (\Omega ; \BC^n))} \leq C \qquad (t > 0)
\end{align*}
delivers~\eqref{Eq: Lp-Lq estimates} for $2 \leq p \leq q$. Using the semigroup law $T(t) = T(t / 2) T(t / 2)$ together with~\eqref{Eq: Lp-Lq for the semigroup restricted} then delivers the estimate for the desired range of numbers $p$ and $q$. \qed

\subsection*{Proof of Theorem~\ref{Thm: Navier-Stokes}}

\noindent The existence part follows by the usual iteration scheme. Notice, that in the classical literature, see, \textit{e.g.},~\cite{Giga, K}, it is required that the semigroup satisfies $\L^p$-$\L^q$-estimates and gradient estimates in $\L^3$. This is especially used by Kato in~\cite{K}. In particular, he obtains bounds on the gradient of the solution to the Navier--Stokes equations. However, if one is only interested into a construction of solutions to the Navier--Stokes equations with the properties of Theorem~\ref{Thm: Navier-Stokes}, \textit{i.e.}, without a control on the gradient of the solution, then one can perform the iteration scheme carried out by Giga~\cite{Giga}. Notice that this iteration scheme can be carried out with the weaker estimates proven in Theorem~\ref{th-1.1}. However, this is not stated in~\cite{Giga} but is presented in detail in~\cite[Sec.~6.3]{Tol-phd}. \qed

\end{document}